\documentclass[10pt]{amsart}

\usepackage{amsmath}
\usepackage{amsfonts}
\usepackage{amssymb}
\usepackage{amsthm}
\usepackage{comment}
\usepackage{epsfig}
\usepackage{psfrag}
\usepackage{mathrsfs}
\usepackage{amscd}
\usepackage[all]{xy}
\usepackage{rotating}
\usepackage{lscape}
\usepackage{amsbsy}
\usepackage{verbatim}
\usepackage{moreverb}
\usepackage{sseq}
\usepackage{pdfpages}
\usepackage{enumerate}
\usepackage{tikz-cd}
\usetikzlibrary{decorations.pathmorphing}
\usepackage{float}
\restylefloat{table}
\usepackage{centernot, mathtools, stmaryrd}

\usepackage{hyperref}
\hypersetup{%
  bookmarksnumbered=true,%
  bookmarks=true,%
  colorlinks=true,%
  linkcolor=black,%
  citecolor=black,%
  filecolor=black,%
  menucolor=black,%
  pagecolor=black,%
  urlcolor=black,%
  pdfnewwindow=black,%
  pdfstartview=FitBH}

\usepackage{url}
\usepackage{todonotes}
\newcommand{\wt}[1]{\widetilde{#1}}

\newcommand{\BPCn}{BP^{(\!(C_{2^n})\!)}}
\newcommand{\BPCnone}{BP^{(\!(C_{2^{n-1}})\!)}}
\newcommand{\MUCn}{MU^{(\!(C_{2^n})\!)}}

\newcommand{\MUG}{MU^{(\!(G)\!)}}
\newcommand{\BPG}{BP^{(\!(G)\!)}}
\newcommand{\BPCeight}{BP^{(\!(C_8)\!)}}
\newcommand{\MUCeight}{MU^{(\!(C_8)\!)}}
\newcommand{\MUR}{MU_\mathbb{R}}
\newcommand{\BPR}{BP_\mathbb{R}}
\newcommand{\MU}{MU}

\DeclareMathOperator{\SliceSS}{\text{SliceSS}}
\DeclareMathOperator{\HFPSS}{\text{HFPSS}}

\newtheorem{theorem}{Theorem}[section]
\newtheorem*{theorem*}{Theorem}
\newtheorem*{claim*}{Claim}
\newtheorem{proposition}[theorem]{Proposition}
\newtheorem{lemma}[theorem]{Lemma}
\newtheorem{cor}[theorem]{Corollary}

\newcommand{\Rc}{\mathcal{R}}
\newcommand{\Rn}{\mathcal{R}_n}
\newcommand{\Fc}{\mathcal{F}}

\newcommand{\Rnone}{\mathcal{R}_{n-1}}
\newcommand{\Rone}{\mathcal{R}_1}

\newcommand{\Rr}{\mathcal{R}_r}
\newcommand{\Rnper}{\mathcal{R}_n^{\per}}
\newcommand{\Roneper}{\mathcal{R}_1^{\per}}

\theoremstyle{definition}
\newtheorem{definition}[theorem]{Definition}

\newtheorem{notation}[theorem]{Notation}

\newtheorem{remark}[theorem]{Remark}

\newcommand{\id}{\mathrm{id}}

\newcommand{\Tor}{\mathrm{Tor}}

\newcommand{\Gal}{\mathrm{Gal}}

\DeclareMathOperator{\smsh}{\wedge}
\newcommand{\G}{\mathbb{G}}

\DeclareMathOperator{\Aut}{Aut}

\newcommand{\Z}{\mathbb{Z}}

\newcommand{\R}{\mathbb{R}}
\newcommand{\F}{\mathbb{F}}

\newcommand{\cM}{\mathcal{M}}

\newcommand{\Q}{{\mathbb{Q}}}
\renewcommand{\to}{\longrightarrow}

\newcommand{\per}{\mathrm{per}}

\newcommand{\Ckm}{C(k,m)}
\newcommand{\Gkm}{G(k,m)}
\newcommand{\Gk}{G(k)}
\newcommand{\Rkm}{R(k,m)}
\newcommand{\Rk}{R(k)}
\newcommand{\Rnm}{\mathcal{R}_n\langle m\rangle}

\usepackage{soul}

\numberwithin{equation}{section}

\author[Beaudry]{Agn\`es Beaudry}
\address{Department of Mathematics, University of Colorado Boulder, Campus Box 395, Boulder, CO 80309-0395}
\author[Hill]{Michael A. Hill}
\address{Department of Mathematics, University of California Los Angeles, 520 Portola Plaza, Los Angeles, CA 90095}
\author[Shi]{XiaoLin Danny Shi}
\address{Department of Mathematics, University of Chicago, 5734 S University Ave, Chicago, IL 60637}
\author[Zeng]{Mingcong Zeng}
\address{Mathematical institute, Utrecht University, Budapestlaan 6, 3584 CD Utrecht, The Netherlands}

\thanks{This material is based upon work supported by the National Science Foundation under grants No.~DMS--1811189 and DMS--1906227.
The authors also thank the Isaac Newton Institute for Mathematical Sciences for support and hospitality during the program Homotopy Harnessing Higher Structures. This work was supported by EPSRC Grant Number EP/R014604/1.}

\title{Models of Lubin--Tate spectra via Real bordism theory}
\date{}

\begin{document}

\begin{abstract}
We study certain formal group laws equipped with an action of the cyclic group of order a power of $2$. We construct $C_{2^n}$-equivariant Real oriented models of Lubin--Tate spectra $E_h$ at heights $h=2^{n-1}m$ and give explicit formulas of the $C_{2^n}$-action on their coefficient rings.
Our construction utilizes equivariant formal group laws associated with the norms of the Real bordism theory $MU_{\mathbb{R}}$, and our work examines the height of the formal group laws of the Hill--Hopkins--Ravenel norms of  $MU_{\mathbb{R}}$.
\end{abstract}
\maketitle
\setcounter{tocdepth}{1}
\tableofcontents



\section{Introduction}\label{sec:intro}
\subsection{Overview.} The results of this paper fall into two categories. First, we prove a series of algebraic results on formal group laws with group actions.  Then, we use these results to construct equivariant refinements of spectra that play a central role in chromatic homotopy theory.
\begin{itemize}
\item \emph{Algebraic results.} We study a family of formal group laws $\Gamma_h$ of height $h=2^{n-1}m$ over a finite field $k$ of characteristic $2$ and certain universal deformations $F_h$. Such formal group laws come equipped with an action of $C_{2^n}$ and we describe deformation parameters that allow for an explicit description of the action of $C_{2^n}$ on the associated Lubin--Tate ring. This reproduces an unpublished result of Hill--Hopkins--Ravenel \cite{OWS} which presents the Lubin--Tate ring with its structure of a $C_{2^n}$-module as the completion of a periodization of the symmetric algebra on a sum of copies of the induced sign representation of $C_{2^n}$. We also incorporate the action of roots of unity in $k$ as part of our narrative.
\item \emph{Topological results.} With the formal group law $\Gamma_h$ and its universal deformation $F_h$ in hand, we obtain a Lubin--Tate spectrum $E(k,\Gamma_h)$. These are ``higher chromatic analogues'' of topological $K$-theory and the group actions we study are higher analogues of Adams operations. Our approach automatically gives us formulas for the action of $C_{2^n}$ on $\pi_*E(k,\Gamma_h)$. We then upgrade $E(k,\Gamma_h)$ to an equivariant spectrum receiving a map from $N_{C_2}^{C_{2^n}} BP_{\R}$, the Hill--Hopkins--Ravenel norm of the Real Brown--Peterson spectrum $BP_{\R}$. The effect of the map on underlying homotopy groups is clear and this opens the door for generalizations of \cite{HHRC4, c4e4} which use the slice spectral sequence to study the homotopy fixed points of $E(k,\Gamma_h)$. These homotopy fixed point spectra play a central role in the chromatic approach to stable homotopy theory. Next, we study the chromatic height of truncations of $N_{C_2}^{C_{2^n}} BP_{\R}$. Although we do not make this precise in this paper, this is data necessary for a study of the height filtration of a ``moduli stack of formal group laws equipped with group actions''.
\end{itemize}
In the rest of this introduction, we will describe our motivation for this project (which comes from homotopy theory) and state our main results.

\subsection{Motivation}

Topological $K$-theory is a remarkably useful cohomology theory that has produced important homotopy-theoretic invariants in topology.
Deep facts in topology have been proved using topological $K$-theory.  Most famously perhaps, Adams used the $K$-theory of real projective spaces together with the action of the Adams operations to resolve the vector fields on spheres problem \cite{AdamsVectorFields}.  Adams and Atiyah also used $K$-theory to give a simpler solution to the Hopf invariant one problem \cite{AdamsAtiyahHopfInvariant}, first solved by Adams in \cite{AdamsHopfInvariant}.

Atiyah \cite{AtiyahKR} describes a version of vector bundles motivated by Galois descent for \(\mathbb C\) over \(\mathbb R\).  The complex conjugation action on complex vector bundles induces a natural $C_2$-action on $KU$. In fact, this corresponds to the action of the Adams operations $(\pm 1)$. Under this action, the $C_2$-homotopy fixed points of $KU$ is $KO$.
Furthermore, there is a homotopy fixed points spectral sequence computing the homotopy groups of $KO$, starting from the action of $C_2$ on the homotopy groups of $KU$.
The spectrum $KU$, equipped with this $C_2$-action and considered as a $C_2$-spectrum, is called  Real $K$-theory $K_\mathbb{R}$.

The main topic of this paper is to construct generalizations of $K_{\mathbb{R}}$, namely Lubin--Tate theories, with explicit actions of higher Adams operations. Our construction of these theories and actions is inspired by the work of Hill, Hopkins and Ravenel \cite{HHR} and makes heavy use of the Real bordism spectrum, which we now introduce.

Conner--Floyd connected complex \(K\)-theory to complex cobordism \(MU^\ast\) \cite{ConnerFloyd}, showing that the Todd genus induces an isomorphism
\[
K^\ast(X)\cong MU^\ast(X)\otimes_{MU^\ast}\mathbb Z.
\]
This refines to a map of spectra \(MU\to KU\).

Early work on $MU$ due to Milnor \cite{Milnor60}, Novikov \cite{Novikov60, Novikov62, Novikov67}, and Quillen \cite{Quillen69} established the complex bordism spectrum as a critical tool in modern stable homotopy theory, with connections to algebraic geometry and number theory through the theory of formal groups. In this language, the map \(MU\to KU\) classifies the multiplicative formal group over \(\mathbb Z\).

Analogously as in the case of $KU$, the complex conjugation action on complex manifolds induces a natural $C_2$-action on $MU$.  This action produces the Real bordism spectrum $MU_\mathbb{R}$ of Landweber \cite{LandweberMUR} and Fujii \cite{FujiiMUR}, studied extensively by Araki \cite{Araki} and by Hu--Kriz \cite{HuKriz}. The underlying spectrum of $MU_\mathbb{R}$ is $MU$, with the $C_2$-action given by complex conjugation.

Complex conjugation acts on $KU$ and $MU$ by coherently commutative ($E_{\infty}$) maps, making $K_\mathbb{R}$ and $MU_\mathbb{R}$ commutative $C_2$-spectra.  The Conner--Floyd map is compatible with the complex conjugation action, and it can be refined to a \textit{Real orientation}, i.e., a $C_2$-equivariant ring map
\[
MU_\mathbb{R} \longrightarrow K_\mathbb{R}.
\]

The spectrum $MU_\mathbb{R} $ is at the roots of the techniques used in Hill, Hopkins and Ravenel's solution to the Kervaire invariant one problem  \cite{HHR}. Since the appearance of these results, there has been an incredible amount of development in equivariant stable homotopy theory.

The techniques of Hill--Hopkins--Ravenel are intimately tied to a subfield of homotopy theory called chromatic homotopy theory.
Chromatic homotopy theory is a powerful tool which studies periodic phenomena in stable homotopy theory by analyzing the algebraic geometry of smooth one-parameter formal groups.  More precisely, the moduli stack of formal groups has a stratification by height $h\geq 0$, which corresponds in the stable homotopy category to localizations with respect to generalizations of the complex $K$-theory spectrum. These are the Lubin--Tate theories $E_h$, also often called the Morava $E$-theories.

As the height increases, this stratification carries increasingly more information about the stable homotopy category, but also becomes harder to understand. Therefore, it is crucial to study higher structures of these spectra, for example, the associated cohomology operations. At all heights $h$, there is a group of cohomology operations generalizing the stable Adams operations on $p$-completed $K$-theory. This group is called the Morava stabilizer group  $\mathbb{G}_h$.

In this paper, we focus our attention at the prime $p=2$ and study the height of spectra obtained from the Hill--Hopkins--Ravenel norms of $MU_{\R}$. Using this, we construct equivariant Real oriented models of  Lubin--Tate spectra $E_h$ with explicit formulas for the actions of finite subgroups of $\mathbb{G}_h$ on their coefficient rings. This is the input needed to determine the $E_2$-pages of the corresponding homotopy fixed points spectral sequences,  which in turn compute the homotopy groups of higher real $K$-theory spectra. These are periodic spectra that generalize the real $K$-theory spectrum $KO$. The connection between our Lubin--Tate theories and $MU_{\R}$ also provides information about differentials in the homotopy fixed points spectral sequences.

Periodic spectra such as the higher real $K$-theories also play a central role in modern detection theorems. These are results about families in the stable homotopy groups of spheres obtained by studying the Hurewicz homomorphisms of these periodic spectra \cite{RavenelOdd,HHR,LiShiWangXu}.

More specifically, let $(k, \Gamma_h)$ be the pair consisting of a finite field $k$ of characteristic 2 and a fixed height-$h$ formal group law $\Gamma_h$ defined over $k$.  Lubin and Tate \cite{lubintate} showed that the pair $(k, \Gamma_h)$ admits a universal deformation $F_h$ defined over a complete local ring with residue field $k$.  This ring is abstractly isomorphic to
\begin{align}\label{eq:classicLTring}
R(k, \Gamma_h):=W(k) [\![ u_1, \ldots, u_{h-1}]\!][u^{\pm 1}].\end{align}
Here, $W(k)$ is the 2-typical Witt vectors of $k$, $|u_i| = 0$, and $|u|=2$.

The Morava stabilizer group $\mathbb{G}(k, \Gamma_h)$ is the group of automorphisms of $(k, \Gamma_h)$ (Definition~\ref{defn:moravastab}).  By the universality of the deformation $(R(k, \Gamma_h), F_h)$ and naturality, there is an action of $\mathbb{G}(k, \Gamma_h)$ on $R(k, \Gamma_h)$.

The group $\mathbb{G}(k, \Gamma_h)$ always contains a subgroup of order two, corresponding to the automorphism $[-1]_{\Gamma_h}(x)$ of $\Gamma_h$.  This $C_2$ subgroup is central in $\mathbb{G}(k, \Gamma_h)$.  Hewett~\cite{Hewett} showed that if $h = 2^{n-1} m$, then there is a subgroup of the Morava stabilizer group isomorphic to $C_{2^n}$ that contains this central $C_2$ subgroup.  Furthermore, if $m$ is odd, then this $C_{2^n}$-subgroup is a maximal finite 2-subgroup in $\mathbb{G}(k, \Gamma_h)$.  See also \cite{bujard}.

The formal group law $F_h$ is classified by a map
\[MU_* \longrightarrow R(k, \Gamma_h),\]
which is Landweber exact (\cite[Section 5]{rezk}).  A Lubin--Tate spectrum $E(k, \Gamma_h)$ is a complex oriented ring spectrum with $\pi_*E(k, \Gamma_h)   =R(k, \Gamma_h)$ whose formal group law is $F_h$.  Topologically, the action of $\mathbb{G}(k, \Gamma_h)$ on $\pi_*E(k, \Gamma_h)$ can be lifted as well.  The Goerss--Hopkins--Miller theorem \cite{rezk, GHMTheorem} shows that $E(k, \Gamma_h)$ is a complex orientable $E_{\infty}$-ring spectrum with a \emph{continuous} action of $\mathbb{G}(k, \Gamma_h)$ by maps of $E_{\infty}$-ring spectra which refines the action of $\mathbb{G}(k, \Gamma_h)$ on $\pi_*E(k, \Gamma_h)$. By a continuous action here, we mean in the sense of Devinatz--Hopkins \cite{DevinatzHopkins,BBGS}.

Now, let $G$ be a finite subgroup of $\mathbb{G}(k, \Gamma_h)$. Classically, the homotopy fixed points spectrum $E(k, \Gamma_h)^{hG}$ is computed by using the homotopy fixed points spectral sequence.  However, at height $h > 2$ (and $p=2$), the spectrum $E(k, \Gamma_h)^{hG}$ is very difficult to compute: given an arbitrary Lubin--Tate spectrum $E(k, \Gamma_h)$, a general formula describing the action of $G$ on $\pi_*E(k, \Gamma_h)$ is not known.  As a result of this, it is hard to compute the $E_2$-page of its homotopy fixed points spectral sequence.  Even worse, the $G$-action on the spectrum $E(k, \Gamma_h)$ is constructed purely from obstruction theory \cite{rezk, GHMTheorem}, so there is no systematic method to compute differentials in the homotopy fixed points spectral sequence.

A major motivation for our work, which arises in \cite{c4e4, e2c4}, is to construct models of Lubin--Tate spectra as equivariant spectra with explicit group actions.  Our construction presents $\pi_* E(k, \Gamma_h)$ explicitly as an $C_{2^n}$-algebra.  As a result, our construction renders the spectra $E(k, \Gamma_h)^{hC_{2^n}}$ accessible to computations via equivariant techniques developed by Hill, Hopkins, and Ravenel \cite{HHR}.

\subsection{Main results}\label{sec:mainresults}
The motivation behind our constructions of Lubin--Tate spectra are equivariant spectra constructed by Hill, Hopkins, and Ravenel in their solution of the Kervaire invariant one problem \cite{HHR}.  To give precise statements of our results and to motivate our proofs, we recall some constructions from \cite{HHR} and use this as an opportunity to introduce some of our notations.

A key construction in Hill--Hopkins--Ravenel's proof of the Kervaire invariant one problem is the detecting spectrum $\Omega$. This spectrum detects all the Kervaire invariant elements in the sense that if $\theta_j \in \pi_{2^{j+1}-2} S^0$ is an element of Kervaire invariant 1, then the Hurewicz image of $\theta_j$ under the map $\pi_*S^0 \to \pi_*\Omega$ is nonzero (see also \cite{HaynesKervaire, HHRCDM1, HHRCDM2} for surveys on the result).

The detecting spectrum $\Omega$ is constructed using equivariant homotopy theory as the fixed points of a $C_8$-spectrum $\Omega_\mathbb{O}$, which in turn is a chromatic-type localization of $\MUCeight := N_{C_2}^{C_8} \MUR$.  Here, $N_{C_2}^{C_8}(-)$ is the Hill--Hopkins--Ravenel norm functor.

Let $\BPR$ be the Real Brown--Peterson spectrum, obtained from the Real bordism spectrum $\MUR$ by the Quillen idempotent (see \cite[Theorem~2.33]{HuKriz} and \cite[Theorem 7.14]{Araki}).  Let $\mathcal{F}$ be the universal 2-typical formal group law over $\pi_*^e \BPR = \pi_*BP$. For $C_{2^n}$ the cyclic group of order $2^n$ with generator $\gamma_n$ we can form the spectrum
$$\BPCn := N_{C_2}^{C_{2^n}} \BPR,$$
and we let
$$\Rn := \pi_*^e \BPCn.$$

For $n \geq 2$, the group $C_{2^n}$ contains a unique subgroup of order $2^{n-1}$ whose generator we will call $\gamma_{n-1}:= \gamma_n^2$.  The maps
\[\eta_L \colon \BPCnone \to i^*_{C_{2^{n-1}}} \BPCn \simeq \BPCnone \smsh \BPCnone\]
induce ring inclusions
\begin{equation}\label{eq:incl}
\begin{tikzcd}
\Rnone \ar[r, hook]& \Rn
\end{tikzcd}
 \end{equation}
which are equivariant with respect to the $C_{2^{n-1}}$-action.  The formal group law $\Fc$, which is originally defined over $\Rone$, can also be viewed as a formal group law over $\Rn$ for each $n \geq 1$ via the ring inclusions $\Rone \hookrightarrow \Rn$.

In fact, we can use the formal group law $\Fc$ to specify generators for $\Rn$ as follows.  For a group $G$ acting on a ring $R$, we use
\[f_g \colon R \to R\]
to denote the ring automorphism that specifies the action of $g$ on $R$. 
For every $n$, there is a canonical strict isomorphism
\[\psi_{\gamma_{n}} \colon \Fc \to \Fc^{\gamma_n}\]
where $\Fc^{\gamma_n}$ is the formal group law obtained from $\Fc$ by applying the automorphism $f_{\gamma_n}$ of $\Rn$ to the coefficients of $\Fc$.  
Since the formal group laws are $2$-typical, the strict isomorphism $\psi_{\gamma_{n}}$ admits the form
\begin{align}\label{eq:psi}
\psi_{\gamma_{n}}(x) = x + \sum_{i\geq 1}{}^{\Fc^{\gamma_{n} }} t_i^{C_{2^n}}x^{2^i},
\end{align}
where $t_i^{C_{2^n}} \in \pi_{2(2^i-1)}^e \BPCn$ for all $i \geq 1$.

It follows from \cite[Section~5.4]{HHR} that the elements $t_i^{C_{2^n}}$, $i \geq 1$ form a set of $C_{2^n}$-algebra generators for $\Rn$.  More precisely, as a $C_{2^n}$-algebra,
\[\Rn  \cong \Z_{(2)} [C_{2^n} \cdot t_1^{C_{2^n}}, C_{2^n} \cdot t_2^{C_{2^n}}, \ldots]\]
where the notation $C_{2^n} \cdot x$ represents the set
\[C_{2^n} \cdot x:= \{x, \gamma_n x, \gamma_n^2 x, \gamma_n^3 x, \ldots, \gamma_n^{2^{n-1}-1}x\}\] with $2^{n-1}$ elements whose degrees are all equal to $|x|$.

The $C_{2^n}$-action on the generators $t_i^{C_{2^n}}$ is specified by the formulas
\[f_{\gamma_n} ( \gamma_n^j t_i^{C_{2^n}})= \left\{\begin{array}{ll} \gamma_n^{j+1} t_i^{C_{2^n}} & j < 2^{n-1} -1 \\
-t_i^{C_{2^n}} & j = 2^{n-1} -1. \end{array} \right. \]
Using the inclusions \eqref{eq:incl}, we may view the generators $C_{2^r} \cdot t_i^{C_{2^r}}$ of $\mathcal{R}_r$ as elements of $\Rn$ for every $r \leq n$.

The underlying spectra of $\MUCn$ and $\BPCn$ are smash products of $2^{n-1}$-copies of $MU$ and $BP$ respectively.  As result of \cite[Proposition~11.28]{HHR}, $C_{2^n}$-equivariant maps from the underlying homotopy of $\MUCn$ (resp. $\BPCn$) to a graded $C_{2^n}$-equivariant commutative ring $R$ are in bijection with formal group laws (resp. 2-typical formal group laws) $F$ over $R$ that are equipped with strict isomorphisms
$$\psi_{\gamma_n^{i+1}}: \Fc^{{\gamma_n^i}} \longrightarrow \Fc^{\gamma_n^{(i+1)}}, \hspace{0.2in} 0 \leq i \leq 2^{n-1}-1$$
with $\psi_{\gamma_n^{i+1}} = (\gamma_n^i)^*\psi_{\gamma_n}$ such that the composition of all the $\psi_{\gamma_n^{i+1}}$'s is the formal inversion on $\Fc$.

Associated to the universal deformation $(E(k, \Gamma_h)_*, F_h)$ and the action of the generator $\gamma_n \in C_{2^n}$, there is a $C_{2^n}$-equivariant map
$$\pi_*^e \BPCn \longrightarrow E(k, \Gamma_h)_*.$$

As is customary, let $v_i \in \pi_{2(2^i-1)} BP$ be the Araki generators, so that
\[{[}2{]}_{\Fc}(x) =  \sum_{i\geq 0}{}^{\Fc} v_i x^{2^i}.\]
Recall that a $2$-typical formal group law $H$ over a $\Z_{(2)}$-algebra $S$ is classified by a map
\[ BP_* \cong \Z_{(2)}[v_1, v_2, \ldots] \to S.\]
It is said to have height $h$ if there is a unit $\lambda \in S$ such that
\[[2]_{H}(x) =  \lambda \cdot x^{p^h} + \text{higher order terms} \hspace{0.1in} \mod(2,v_1, \ldots, v_{h-1}).\]

The main result of this paper arises from an observation of Hill--Hopkins--Ravenel \cite{HHR} that the formal group law $\Fc$ over $\pi_*^e \BPCeight$ should be of height $h=4$ after inverting some carefully chosen element $D$ \cite[Section~11.2]{HHR}. More generally, for
\[h=2^{n-1}m,\]
where $m\geq 1$ is any natural number, then there is an element $D$ so that $\Fc$ has height $h$ over $\pi_*^eD^{-1}BP^{(\!(C_{2^n})\!)}$. In practice, choosing an appropriate element $D$ to invert appears to become tedious when $m$ in $h=2^{n-1}m$ is large.  This is not hard when $n=2$ and $m=1$ \cite[(9.3)]{HHRC4}, but already becomes tricky when $n=2$ and $m=2$ \cite[Theorem 1.1]{c4e4}.

Studying the height of a formal group law over $R$ is done by studying the image of the elements $v_i$ in $R$.
One reason that this is difficult in our case is that the image of the $v_i$'s in $\pi_*^e \BPCn$ are given by intricate formulas in terms of the generators $t_i^{C_{2^n}}$. For example, when $n=2$, $\pi_*^e\BPCn \cong BP_*BP$ and giving formulas for the $v_i$'s is equivalent to giving formulas for the conjugation on the Hopf algebroid. This is essentially the task of giving an analogue of Milnor's formula in $\mathcal{A}_* = (H\F_2)_*H\F_2$ which relates the generators $\xi_i$ with their conjugates, but in the case of $BP$ instead of $H\F_2$. One contribution of this paper and a key result is to give these formulas in terms of explicit, clean recursive relations.
\begin{theorem}\label{theorem:Mform}
For every $n\geq 2$ and $k \geq 1$,
\begin{eqnarray}
t_k^{C_{2^{n-1}}}
&\equiv& t_k^{C_{2^n}} + \gamma_n t_k^{C_{2^n}} + \sum_{j=1}^{k-1}\gamma_{n}{t}_{j}^{C_{2^n}} ({t}_{k-j}^{C_{2^n}})^{2^j} \pmod{I_k}\label{eq:tkrecursionformula}
\end{eqnarray}
where $I_k =(2,v_1, \ldots, v_{k-1})$.
\end{theorem}
Section~\ref{sec:formulas} is dedicated to the proof of this result, which involves a detailed analysis of the relationship between the $v_i$ and the $t_i^{C_{2^r}}$ generators. An important result contained in this section is Proposition~\ref{prop:danny2}, which states that the ideals $I_k \subseteq \Rn$ are preserved by the action of $C_{2^n}$.

With this in hand, we are ready to study the chromatic filtration arising from $BP^{(\!(C_{2^n})\!)}$.  In analogy with the truncation $BP\langle h \rangle$ of the Brown--Peterson spectrum $BP$, one can form equivariant quotients
\[ BP^{(\!(C_{2^n})\!)} \langle m \rangle :=BP^{(\!(C_{2^n})\!)} /( C_{2^n} \cdot \bar t_{m+1},C_{2^n} \cdot \bar t_{m+2}, \ldots  ). \]
The quotient here is done by using the method of twisted monoid rings \cite[Section 2.4]{HHR}.  Let
\[\Rnm := \pi_*^e  BP^{(\!(C_{2^n})\!)} \langle m \rangle \cong \Z_{(2)}[C_{2^n} \cdot t_1^{C_{2^n}},\ldots, C_{2^n}  \cdot  t_m^{C_{2^n}} ].\]
The left unit induces a map
\[ BP_* \to  \Rnm   \]
and so $\Rnm $ carries a formal group law which we will continue to denote by $\Fc$.

Our paper studies the height of the spectra $BP^{(\!(C_{2^n})\!)} \langle m \rangle$ and prove that they are of height $h = 2^{n-1}m$. The first step of our analysis is completely algebraic and consists of studying the image of the $v_i$ generators in $\Rnm $.

\begin{theorem}\label{thm:listofawesomeness}
There is an element $D \in \Rnm $ such that
\begin{itemize}
\item $v_h$ divides $D$ in $\Rnm$,
\item $(2, v_1, v_2, \ldots, v_h)$ is a regular sequence in $D^{-1}\Rnm$,
\item $v_r \in I_r$ for $r>h$,
\item $D^{-1}\Rnm/I_h \cong \F_2[(t_m^{C_{2^n}})^{\pm 1}]$ with $v_h = t_m^{(2^h-1)/(2^m-1)}$, and
\item the formal group law $\Fc$ has height exactly $h$ over $D^{-1}\Rnm/I_h $.
\end{itemize}
\end{theorem}

The results of Theorem~\ref{thm:listofawesomeness}, discussed in Propositions~\ref{prop:IhIGBP} and \ref{prop:heightBPbracketsm} below, are actually proved by passing to a completion of an extension of $\Rnm$.
For $k$ a finite field of characteristic 2, we let $W(k)$ be the ring of Witt vectors on $k$. Let
\begin{align*}
\Rkm:=W(k) [C_{2^n} \cdot t_1^{C_{2^n}}, \ldots, C_{2^n} \cdot t_{m-1}^{C_{2^n}}, C_{2^{n}} \cdot u ] [C_{2^{n}}\cdot u^{-1} ]^{\wedge}_\mathfrak{m}.
\end{align*} where $|t_i^{C_{2^n}}| = 2(2^i-1)$ for $1 \leq i \leq m-1$ and $|u| =2$.  The ideal $\mathfrak{m}$ is given by
\[ \mathfrak{m}=(C_{2^n}\cdot t_1^{C_{2^n}}, \ldots, C_{2^n}\cdot t_{m-1}^{C_{2^n}}, C_{2^n}\cdot (u-\gamma_n u)  ).\]
Note that the ring $\Rkm$ is an unramified extension of a completion of the ring $\Rnm[C_{2^n} \cdot (t_m^{C_{2^n}})^{-1}]$. The action of $C_{2^n}$ on $\Rkm$ is dictated by the notation.

The advantage of working in $\Rkm$ is that it is a complete local ring whose Krull dimension is easily determined. The following, which is Proposition~\ref{prop:m=Ih} below, is the most technical argument of the paper and is the main goal of Section~\ref{sec:proofs-alg}.
\begin{theorem}\label{thm:keyintrotricky}
In $\Rkm$, the ideal $I_h = (2, v_1, \ldots, v_{h-1})$ is equal to the the maximal ideal $\mathfrak{m} = (C_{2^n} \cdot t_1^{C_{2^n}}, \ldots, C_{2^n} \cdot t_{m-1}^{C_{2^n}}, C_{2^n}\cdot (u-\gamma_n u))$.
\end{theorem}
Using a Krull dimension argument, this result implies that $(2,v_1, v_2, \ldots)$ forms a regular sequence in $R(k,m)$. We let
\[\Gamma_h = p^*\Fc\]
where $p \colon \Rkm \to \Rkm/I_h \cong K:=k[u^{\pm 1}]$ is the quotient map. The following corollary is an immediate consequence of Theorem~\ref{thm:keyintrotricky}.
\begin{cor}\label{cor:introCorUnivDef}
The pair $(\Rkm, \Fc)$ is a universal deformation of $(K, \Gamma_h)$.
\end{cor}
\noindent
In other words, $\Rkm$ is the Lubin-Tate ring of a universal deformation of the formal group law $\Gamma_h$ of height $h$.

Once Theorem~\ref{thm:keyintrotricky} and Corollary~\ref{cor:introCorUnivDef} have been established, we begin to upgrade our results to homotopy theory via the Landweber Exact Functor Theorem. This gives rise to a complex orientable cohomology theory
\[X \longmapsto \Rkm\otimes_{BP_*}BP_*(X) \]
represented by a spectrum whose homotopy groups are $\Rkm$. We call the representing spectrum $E(k, \Gamma_h)$, as it is a form of the Lubin--Tate spectrum.

Before stating this as one of our results, we introduce more details with respect to the natural group actions at hand. Let
\begin{align*}
q &= 2^m-1
\end{align*}
and $k^\times[q] \subseteq k^\times$ be the subgroup of $q$-torsion elements. This is the subgroup of elements $\zeta \in k^\times$ so that $\zeta^q = 1$.  We let
\[\Gal := \Gal(k/ \mathbb{F}_2)\]
and
\[\Ckm := \Gal \ltimes k^\times[q]\]
where the action of $\Gal$ on $k^\times[q]$ is the natural action of the Galois group.
Let $\Gkm$ denote the group
\[\Gkm := C_{2^n} \times (\Gal \ltimes k^\times[q]).\]
We note that $\Gkm$ depends on $n$ also, but we think of $n$ as being fixed while allowing $m$ and $k$ to vary.

There is an obvious action of $\Gkm$ on $\Rkm$ defined as follows:
The action of $C_{2^n}$ on $\Rkm$ is the $W(k)$-linear action determined by
\begin{align}
\label{eq:action1}
f_{\gamma_n}(\gamma_n^r x) = \left\{ \begin{array}{ll} \gamma_n^{r+1} x & r < 2^{n-1}-1 \\
-x & r = 2^{n-1}-1
\end{array}\right. \end{align}
for $x=t_i^{C_{2^n}}$ ($1\leq i\leq m-1$) and $x=u$. The group $\Gal(k/\mathbb{F}_2)$ acts on $\Rkm$ via its action on the coefficients $W(k)$.  The group $k^\times[q]$ acts on $\Rkm$ by
 \begin{align} \label{eq:action2}
f_{\zeta}(u)&=\zeta^{-1} u,
\end{align}
and
 \begin{align} \label{eq:action3}
f_{\zeta}(t_i^{C_{2^n}}) &= t_i^{C_{2^n}}
\end{align}
for every $\zeta \in k^{\times}[q]$ and $1 \leq i \leq m-1$.  All together, these three actions combine to give an action of $\Gkm$ on $\Rkm$.

We can now state our next main result, which is also one of the main motivation for our work, and is proved in Section~\ref{sec:classLT}.
\begin{theorem}\label{thm:modelE0}\label{thm:model}
There exists a height-$h$ formal group law $\Gamma_h$ defined over $\F_2$ such that for any finite field $k$ of characteristic 2, there is a Lubin--Tate theory $E(k, \Gamma_h)$, functorial in $k$, such that
\begin{equation}\label{eq:isointro}\pi_*E(k,\Gamma_h) \cong \Rkm.\end{equation}
Furthermore, there is a subgroup $\Gkm$ inside the Morava stabilizer group $\mathbb{G}(k, \Gamma_h)$ so that the isomorphism \eqref{eq:isointro} is equivariant for the action of $\Gkm$ where the action of $\Gkm$ on $\Rkm$ is described in \eqref{eq:action1}, \eqref{eq:action2}, and \eqref{eq:action3}.
\end{theorem}

\begin{remark}\label{rem:pi0remark}
The ring $\pi_*E(k, \Gamma_h) \cong \Rkm$ is usually described in a way that emphasizes the structure of $\pi_0E(k, \Gamma_h) \cong \Rkm_0$ as in Equation~\ref{eq:classicLTring} above (where $R(k, \Gamma_h) \cong \pi_*E(k,\Gamma_h)$), that is, as a power series ring over $W(k)$ on $h-1$ deformation parameters $u_1, \ldots, u_{h-1}$.
In order to give a more familiar description of $\pi_*E(k, \Gamma_h)$, in Proposition~\ref{thm:modelE0-alg}, we give a description of $\Rkm$ which displays the structure of $\Rkm_0$. We show that there are elements
\[C_{2^n}\cdot \tau_i = \{\tau_i, \gamma_n \tau_i, \ldots, \gamma_{n}^{2^{n-1}-1}\tau_i \} \subseteq \Rkm_0\]
for $1\leq i\leq m-1$ and
 \[\overline{C_{2^n}\cdot \tau_m} = \{\tau_m, \gamma_n \tau_m, \ldots, \gamma_{n}^{2^{n-1}-2}\tau_m \} \subseteq \pi_0E(k, \Gamma_h)\]
such that
\[ \Rkm \cong  W(k) [\![ C_{2^n}\cdot \tau_1, \ldots, C_{2^n}\cdot \tau_{m-1},\overline{C_{2^n}\cdot \tau_m} ]\!][u^{\pm1}] .\]
Note that there are $h-1$ power series generators in this description of $\Rkm$ and once we realize $\Rkm_0$ as the Lubin-Tate ring for the universal deformation of a formal group law of height $h$, these will be deformation parameters.
We also describe the action of $\Gkm$ explicitly in terms of these generators in Proposition~\ref{thm:modelE0-alg}.
\end{remark}

Once Theorem~\ref{thm:model} has been established, given that its construction was motivated by the spectra $BP^{(\!(C_{2^n})\!)}\langle m\rangle$ which are equivariant spectra, it is natural to seek to refine $E(k,\Gamma_h)$ to an equivariant spectrum.
As discussed earlier, by the Goerss--Hopkins--Miller theorem \cite{rezk, GHMTheorem}, $E(k, \Gamma_h)$ is a complex orientable $E_{\infty}$-ring spectrum with a continuous action of $\mathbb{G}(k, \Gamma_h)$ by maps of $E_{\infty}$-ring spectra which refines the action of $\mathbb{G}(k, \Gamma_h)$ on $\pi_*E(k, \Gamma_h) \cong \Rkm$.  In other words, for any $G \subseteq \mathbb{G}(k, \Gamma_h)$, we may view $E(k, \Gamma_h)$ as a commutative ring object in naive \(G\)-spectra.

The functor
\[X\longmapsto F(EG_{+},X)\]
takes naive equivalences to genuine equivariant equivalences, and hence allows us to view $E(k, \Gamma_h)$ as a genuine \(G\)-equivariant spectrum. The commutative ring spectrum structure on $E(k, \Gamma_h)$ gives an action of a trivial \(E_{\infty}\)-operad on the spectrum \(F(EG_{+},E(k, \Gamma_h))\).  Work of Blumberg--Hill \cite{BlumbergHill} shows that this is sufficient to ensure that \(F(EG_{+},E(k, \Gamma_h))\) is actually a genuine equivariant commutative ring spectrum, and hence it has norm maps.

The spectrum $E(k, \Gamma_h)$ has an action of $\Gkm$ by maps of $E_{\infty}$-ring spectra that refines the $\Gkm$-action on $\pi_*E(k, \Gamma_h)$ described in Theorem~\ref{thm:model} (or rather, in Equations~\eqref{eq:action1}, \eqref{eq:action2}, and \eqref{eq:action3}).  By passing to the cofree localization $F(E{C_{2^n}}_+, E(k, \Gamma_h))$, we may view $E(k, \Gamma_h)$ as a commutative $C_{2^n}$-spectrum.

Recent work of Hahn--Shi \cite{hahnshi} establishes the first known connection between the obstruction-theoretic actions on Lubin--Tate theories and the geometry of complex conjugation.  More specifically, there is a Real orientation for any of the $E(k, \Gamma_h)$: there is a $C_2$-equivariant homotopy commutative ring map
\[
\MU_{\mathbb R}\longrightarrow i_{C_{2}}^{\ast}E(k, \Gamma_h).
\]
Using the norm-forget adjunction, for any finite $G \subset \mathbb{G}(k, \Gamma_h)$ that contains the central $C_2$-subgroup, there is a $G$-equivariant homotopy commutative ring map
\[
\MUG\longrightarrow N_{C_{2}}^{G}i_{C_{2}}^{\ast}E(k, \Gamma_h) \longrightarrow E(k, \Gamma_h).
\]

For our explicit forms of $E(k, \Gamma_h)$, we have the following theorem, which is proved at the beginning of Section~\ref{sec:equivspectrum}.

\begin{theorem}\label{thm:IntroEquivOrientation}
There is a $C_{2^n}$-equivariant homotopy commutative ring map
$$\MUCn \longrightarrow E(k,\Gamma_h).$$
This map factors through a homotopy commutative ring map
\[ \phi \colon BP^{(\!( C_{2^n})\!)} \to E(k,\Gamma_h)\]
such that $\pi_*^e\phi$ is the map
$ \Rn \to \pi_*E(k, \Gamma_h)$
determined by
\begin{align*}
t_i^{C_{2^n}} &\longmapsto \begin{cases}t_i^{C_{2^n}} & 1\leq i \leq m-1, \\
 u^{2^m-1} & i=m, \\
0 & i>m.
\end{cases}
\end{align*}
\end{theorem}

The $C_{2^n}$-equivariant spectra $\MUCn$ and $\BPCn$ are accessible to computations, and the existence of equivariant orientations renders computations that rely on the slice spectral sequence tractable.  Using differentials in the slice spectral sequence of $\MUR$ and the Real orientation $\MUR \to E(k, \Gamma_h)$, Hahn and Shi \cite[Theorem~1.2]{hahnshi} computed $E(k, \Gamma_h)^{hC_2}$, valid for arbitrarily large heights $h$.  Computations of \cite[Section~9]{HHR}, \cite{HHRC4}, \cite{c4e4}, and \cite{e2c4} show that there are systematic ways of obtaining differentials in the slice spectral sequences of $\MUCn$, $\BPCn$, and their localizations using techniques in equivariant homotopy theory.

We prove the following theorem in Section~\ref{sec:localization}:
\begin{theorem}\label{thm:IntroInvertingD}
There is an element $D \in \pi_{*\rho_{C_{2^n}}}^{C_{2^n}} \MUCn$, where $\rho_{C_{2^n}}$ is the real regular representation,
that becomes invertible under the map $\pi_\bigstar^{C_{2^n}} \MUCn \longrightarrow \pi_\bigstar^{C_{2^n}} E(k, \Gamma_h)$ such that there are factorizations
$$\begin{tikzcd}
\MUCn \ar[d] \ar[r] &E(k, \Gamma_h) \\
D^{-1} \MUCn \ar[ru, dashed]
\end{tikzcd} \hspace{0.3in} \begin{tikzcd}
\BPCn \ar[d] \ar[r] &E(k, \Gamma_h) \\
D^{-1} \BPCn \ar[ru, dashed]
\end{tikzcd}$$
of the $C_{2^n}$-equivariant orientations through $D^{-1}\MUCn$ and $D^{-1}\BPCn$.  Furthermore, the spectra $D^{-1} \MUCn$ and $D^{-1}\BPCn$ are cofree and satisfy the Hill--Hopkins--Ravenel periodicity theorem \cite[Theorem 9.19]{HHR}.
\end{theorem}

When $G = C_8$ and $h =4$, the Hill--Hopkins--Ravenel detecting spectrum $\Omega$ is defined as the $C_8$-fixed points of $\Omega_\mathbb{O} := D_{\text{HHR}}^{-1}\MUCeight$, where
$$D_{\text{HHR}} = (N_{C_2}^{C_8} \bar{r}_1^{C_8})(N_{C_2}^{C_8} \bar{r}_3^{C_4}) (N_{C_2}^{C_8} \bar{r}_{15}^{C_2}) \in \pi_{19\rho_8}^{C_8} \MUCeight$$
is defined in \cite[Section 9]{HHR}.  Our proof of Theorem~\ref{thm:IntroInvertingD} implies that $D$ is divisible by $D_{\text{HHR}}$,
and there is a factorization
$$\begin{tikzcd}
\MUCeight \ar[r] \ar[d] & E(k, \Gamma_4)\\
\Omega_\mathbb{O} \ar[ru, dashed]
\end{tikzcd}$$
of the $C_8$-equivariant orientation of $E(k, \Gamma_4)$ through $\Omega_\mathbb{O}$.

We finally return to our analysis of the chromatic filtration arising from $BP^{(\!(C_{2^n})\!)}$.   We have already discussed some algebraic properties of this filtration in Theorem~\ref{thm:listofawesomeness}. In Section~\ref{sec:heightfilt} we also study the chromatic localizations of $BP^{(\!(C_{2^n})\!)}$. We let $K(r)$ be any form of Morava $K$-theory at height $r$. The Bousfield localization functor $L_{K(r)}$ is independent of this choice. Among other results on localizations, we prove the following result in Theorem~\ref{cor:computingKr}.
\begin{theorem}
The $K(r)$-localization of the spectrum $ i_e^* \BPG \langle m \rangle$ is non-zero for $0 \leq r \leq h$ and trivial when $r>h$.
\end{theorem}

\subsection{Summary of the contents}
We now turn to a summary of the contents of this paper.  Sections~\ref{sec:actions}, \ref{sec:formulas} and \ref{sec:proofs-alg} contain our algebraic results and Sections~\ref{sec:proofs-top}, \ref{sec:localization} and \ref{sec:heightfilt} focus on the topological results.

Section~\ref{sec:actions} provides the necessary background for Lubin--Tate deformation theory and sets up a framework for studying formal group laws with actions of finite groups.  We then use this framework to study formal group laws equipped with a $C_{2^n}$-action that is compatible with certain cyclic groups of orders coprime to 2 and the associated Galois groups.
In Section~\ref{sec:formulas}, we study the deformation parameters  $t_k^{C_r}$ prove recursive formulas relating the $t_k^{C_r}$-generators for various values of $r$.  The key result of this section is Theorem~\ref{theorem:Mform}, which expresses $t_k^{C_{2^{n-1}}}$ in terms of $t_i^{C_{2^n}}$ for every $n \geq 2$ and $k \geq 1$.  These formulas will be essential for proving our main theorems.
In Section~\ref{sec:proofs-alg}, we prove that the maximal ideal of $\mathfrak{m}$ of $R(k,m)$ is equal to $I_h$ and introduce the formal group law $\Gamma_h$.

In Section~\ref{sec:proofs-top}, we construct the spectrum $E(k,\Gamma_h)$. We then refine it to an equivariant spectrum and study some properties of that equivariant theory in Section~\ref{sec:localization}. In particular, we introduce the periodicity generator $D \in \pi_{\star}BP^{(\!(C_{2^n})\!)}$.
In Section~\ref{sec:heightfilt}, we examine the Bousefield localizations of  $BP^{(\!(C_{2^n})\!)}\langle m\rangle $ and of $D^{-1}BP^{(\!(C_{2^n})\!)}\langle m\rangle $ with respect to the Morava $K$-theories $K(r)$.

\subsection{Acknowledgements}
Special cases of these equivariant Real oriented Lubin--Tate theories at small heights have already appeared in \cite{HHR, HHRC4, c4e4, e2c4}, if not explicitly, certainly in spirit. These papers have both inspired and influenced various aspects of this project.
The authors would like to thank Mark Behrens, Irina Bobkova, Paul Goerss, Mike Hopkins, Hana Jia Kong, Peter May, Lennart Meier, Haynes Miller, Doug Ravenel, Vesna Stojanoska, Guozhen Wang, and Zhouli Xu for helpful comments and conversations.


\section{Formal group laws with group actions}\label{sec:actions}
The goal of this section is to set up a framework for formal group laws with a $G$-action.  We then use this framework to study formal group laws equipped with a $C_{2^n}$-action that is compatible with the actions of certain cyclic groups of orders coprime to $2$ and the associated Galois groups.

\subsection{Formal group laws with $C_{2^n}$-action}
In this subsection, we carry out a discussion similar to that of \cite[Section 11.3]{HHR}.
By an \emph{even periodic graded ring} $R$, we mean a graded commutative $\Z_{(2)}$-algebra $R$ such that
\begin{enumerate}
\item $R_{2k+1}=0$ for all $k\in \Z$; and
\item there exists a unit in $R_2$.
\end{enumerate}
Note that condition (2) implies that $R_{2k} \cong R_0$ for all $k\in \Z$.

Consider the category $\mathcal{M}_{FG}^{h,\mathrm{per}}$, whose objects are triples $(R,u, F)$, where
\begin{enumerate}
\item $R$ is an even periodic graded ring;
\item $u\in R_2$ is a unit; and
\item $F$ is a homogenous formal group law of degree $-2$ over $R$.  That is,
\[F(x,y) \in (R[\![x,y]\!])_{-2},\]
where the formal variables $x$ and $y$ both have degree $-2$ in $R[\![x,y]\!]$.
\end{enumerate}
Morphisms in $\mathcal{M}_{FG}^{h,\mathrm{per}}$ are pairs
\[(f, \psi) \colon (R, u, F) \to (S, w, G),\]
where $f \colon R \to S$ is a graded ring homomorphism and
\[\psi \colon G \to f^*F \]
is a strict isomorphism of formal group laws. Note that there is no condition on the morphism relating $u$ and $w$.

Consider also the category $\mathcal{M}^{\mathrm{per}}_{FG}$, whose objects are triples $(R, u , \widetilde{F})$, where
\begin{enumerate}
\item $R$ is an even periodic graded ring;
\item $u \in R_2$ a unit; and
\item $\wt{F}$ is a formal group law of degree $0$ over $R_0$.  That is,
\[ \wt{F}(\wt{x}, \wt{y}) \in R_0[\![\wt{x}, \wt{y}]\!],\]
with $\wt{x}, \wt{y}$ of degree $0$  in $R_0[\![\wt{x}, \wt{y}]\!]$.
\end{enumerate}
Morphisms in this category are pairs
\[(f, \wt{\psi}) \colon (R, u, \wt{F}) \to (S, w, \wt{G})\]
with $f: R \longrightarrow S$ a homomorphism of graded rings and
\[\wt{\psi} \colon \wt{G} \to f^*\wt{F}  \]
an isomorphism of formal group laws, not necessarily strict, such that
\[f(u) = \wt{\psi}'(0)w.\]

There is an isomorphism of categories
\begin{equation}\label{eq:equicats}
\Psi :\xymatrix{ \mathcal{M}_{FG}^{\mathrm{per}} \ar@<1ex>[r]&  \mathcal{M}_{FG}^{h,\mathrm{per}} \ar@<1ex>[l]  } : \Phi,
\end{equation}
given by
\begin{align*}\Psi(R, u, \wt{F}(\wt x,\wt y)) &= (R, u^{-1}\wt{F}(u x ,u y)) , \\ \Psi( f, \wt{\psi}(\wt x)) &= (f,   (\wt{\psi}'(0) w)^{-1}\wt{\psi}(w x) ), \end{align*}
and
\begin{align*}\Phi(R, u, {F}(x,y)) &= (R, u{F}(u^{-1} \wt x ,u^{-1} \wt y)) , \\ \Phi( f, {\psi}(x)) &= (f,  f(u) \psi(w^{-1}\wt x)). \end{align*}

\begin{notation}\label{not:ungradedfgl}
For $(R,u,F) \in \mathcal{M}_{FG}^{h,\mathrm{per}}$, we let
\[\wt F (\wt x , \wt y) :=u F(u^{-1} \wt x, u^{-1}\wt y).\]
\end{notation}

Next, we turn to group actions on formal group laws.
\begin{definition}
An action of a group $G$ on a formal group law $F$ over $R$ is a functor
\[BG \to \cM_{FG}^{h,\mathrm{per}} \]
such that $(R,u,F)$ is the image of the unique object in $BG$.
\end{definition}

\begin{remark}
By the equivalence \eqref{eq:equicats}, an action on a homogenous formal group law in $ \cM_{FG}^{h,\mathrm{per}} $ corresponds to an action on the associated non-homogenous object in $ \cM_{FG}^{\mathrm{per}} $, and vice versa.
\end{remark}

\begin{definition}
For an even periodic graded ring $R$, let $c \colon R \to R$ be the ring isomorphism which is multiplication by $(-1)^{n}$ on $R_{2n}$. We call $c$ the \emph{involution}.
\end{definition}

For $(R,u, F) \in \cM_{FG}^{h,\mathrm{per}}$,
\begin{align*} c^*F &= -F(-x,-y).
\end{align*}
Furthermore, there is an automorphism
\[(c, c(x)) \colon (R, u, F) \to (R, u, F)\]
in $\cM_{FG}^{h,\mathrm{per}}$, where $c(x) = -[-1]_F(x)$. Since $c^2=\id$, this determines a functor
\[BC_2 \xrightarrow{(R,u,F)}\cM_{FG}^{h,\mathrm{per}}, \]
where the unique object in the category $BC_2$ is mapped to $(R, u, F)$ and the non-identity morphism is mapped to $(c, c(x))$.  In other words, every object of $\cM_{FG}^{h,\mathrm{per}}$ comes with a natural $C_2$-action.

\begin{definition}
The action of $C_2$ on $(R,u,F) \in \cM_{FG}^{h, \mathrm{per}}$ where the generator of $C_2$ acts by $(c, c(x))$ is called the \emph{conjugation action}.
\end{definition}

Recall that $\Rn =  \pi_*^e \BPCn$ and we are letting $\Fc$ denote the image of the universal $2$-typical formal group law under the left unit.  Let
\[\Rnper := \Rn[C_{2^n}\cdot u^{\pm 1}],\]
where $C_{2^n}\cdot u \subseteq (\Rn^{\per})_2$.

\begin{proposition}\label{prop:actionC2n}
Let $R$ be an even periodic graded ring with an action of $C_{2^n}$ that restricts to the involution action on $C_2$.  A $C_{2^n}$-equivariant graded ring homomorphism
$$f: \Rnper \longrightarrow R$$
determines a 2-typical formal group law $(R, f(u), f^*\Fc) \in \cM_{FG}^{h,\mathrm{per}}$.  This formal group law is equipped with a $C_{2^n}$-action by maps of 2-typical formal group laws that extends the $C_2$-conjugation action.
\end{proposition}

This follows from \cite[Proposition 11.28]{HHR}, adapted to the $2$-typical, periodic case.
In this section, we abbreviate $\gamma=\gamma_n$ for the generator of $C_{2^n}$. Let
\[(f_\gamma, \psi_{\gamma}(x))  \colon (\Rnper,  u, \Fc ) \to  (\Rnper,u , \Fc )\]
(where $\psi_{\gamma}$ is as in \eqref{eq:psi}) be the action induced by $\gamma$, so that
$$\psi_\gamma: \Fc \longrightarrow \Fc^\gamma = f_\gamma^* \Fc$$
is the strict isomorphism.  We will sometimes abuse notation and write $\gamma = f_\gamma$.  Note that $\Fc$ has coefficients in $BP_* = \Rone \subset \Rn$ and $\psi_{\gamma}$ has coefficients in $\Rn$.

\subsection{Compatible action by roots of unity}
Let $k$ be a finite field of characteristic $2$ and $W(k)$ the ring of Witt vectors over $k$. For a $\Z_{(2)}$-algebra $R$, let
\[\Rk := W(k)\otimes_{\Z_{(2)}} R.\]
Note that there is a unique group homomorphism
\[T \colon k^{\times} \to \Rk_0^{\times},\]
which defines the Teichm\"uller lifts. This allows us to embed $k^{\times}$ in $\Rk$.

There is an action of $k^{\times}$ on $(\Roneper(k), {u}, {\Fc}) \in \cM_{FG}^{h, \mathrm{per}}$ given as follows. Here,
\[\Roneper(k) \cong W(k) [v_1, v_2, \ldots][u^{\pm 1}].\]
Let $\zeta \in k^{\times}$. Then
\[f_\zeta \colon \Roneper(k) \to \Roneper(k)\]
is the $W(k)$-linear map that is determined by
\begin{align}\label{eq:actionfomega}
f_\zeta({u}) &= \zeta^{-1} {u},   \\
f_\zeta({v}_i)&=v_i. \nonumber
\end{align}
Since $\Fc$ is defined over $\Rone$, $f_\zeta^*\Fc=\Fc$. Let
\[{\psi}_{\zeta} \colon \Fc \to  f_\zeta^*\Fc=\Fc\]
be the strict isomorphism given by the identity ${\psi}_{\zeta}(x)=x$. Then the pair $(f_\zeta, {\psi}_{\zeta})$ is an automorphism of $(\Roneper(k), {u}, {\Fc})$ and this defines an action of $k^{\times}$ on the object $(\Roneper(k), u, {\Fc})$ of $ \cM_{FG}^{h, \mathrm{per}}$.

As an immediate consequence, we have:
\begin{proposition}\label{prop:Caction}
Let $R$ be an even periodic graded ring with an action of a subgroup $C\subseteq k^{\times}$. A $C$-equivariant ring homomorphism $f \colon \Roneper(k) \to R$ determines a $2$-typical formal group law $(R, f(u), f^*{\Fc})$ with a $C$-action by maps of $2$-typical formal group laws.
\end{proposition}

\begin{remark}
Under the functor $\Phi \colon \mathcal{M}_{FG}^{h, \per} \to \mathcal{M}_{FG}^{\per}$, the morphism $(f_{\zeta}, \psi_{\zeta}(x))$ maps to $(f_{\zeta}, \wt{\psi}_{\zeta} )$ where $\wt{\psi}_{\zeta}(\wt x) = \zeta^{-1}\wt x$.
\end{remark}

Finally, we want to extend this to an action of $(\Rnper(k), u, {\Fc})$ in a way that commutes with the $C_{2^n}$-action defined in the last subsection. Here, note that
\begin{align}\label{eq:Rnk} \Rnper(k)\cong W(k) [C_{2^n} \cdot t_1^{C_{2^n}}, C_{2^n} \cdot t_2^{C_{2^n}}, \ldots][C_{2^n}\cdot u^{\pm 1}].\end{align}
We will define an action of
\[ C_{2^n} \times k^\times\]
on  $(\Rnper(k), u, {\Fc})$.
First, extend the $C_{2^n}$-action on $\Rnper$ to $\Rnper(k)$ linearly with respect to $W(k)$.
Since elements of $C_{2^n}$ and $k^{\times}$ commute in the product, our definition of the action $f_{\zeta}$ must satisfy
\begin{equation}\label{eq:actionscomm}
f_\zeta(f_{\gamma} (x)) = f_{\gamma}(f_\zeta(x))\end{equation}
for all $\zeta\in k^{\times}$.
We let $f_{\zeta}$ be as in \eqref{eq:actionfomega} on $\Roneper(k) \subseteq \Rnper(k)$.
Also, let
\[f_\zeta(t_i^{C_{2^n}}) = t_i^{C_{2^n}}\]
for all $i \geq 1$. Then, the identity \eqref{eq:actionscomm} determines the action of $C_{2^n} \times k^\times$ on all of $\Rnper(k)$. Note in particular that the $k^{\times}$-action fixes $\Rn \subset \Rnper(k)$.

Since the inclusion $\Roneper(k) \to \Rnper(k)$ is a $k^\times$-equivariant map,
$ (\Rnper(k), u,  {\Fc} )$ inherits a $k^\times$-action from $(\Roneper(k), u, \Fc)$ by Proposition~\ref{prop:Caction}.

\begin{proposition}
The formulas above give an action of $C_{2^n} \times k^\times$  on  the object $(\Rnper(k), u,  {\Fc} )$ of $\mathcal{M}_{FG}^{h,\mathrm{per}}$.
\end{proposition}

\begin{proof}
It remains to verify that $(f_{\zeta}, {\psi}_{\zeta})$ and $(f_{\gamma}, {\psi}_{\gamma})$ commute, i.e., that the morphism
\[(f_\zeta, \psi_\zeta( x))(f_{\gamma}, \psi_\gamma( x) ) = ( f_\zeta f_{\gamma},f_{\zeta}^*\psi_\gamma(  \psi_\zeta( x))), \]
which is the composite
\[ \xymatrix{  \Fc \ar[r]^-{\psi_\zeta} & f_\zeta^*   \Fc \ar[r]^-{f_\zeta^*\psi_{\gamma}} &  f_\zeta^* f_{\gamma}^* \Fc  }, \]
is equal to the morphism
\[(f_\gamma, \psi_\gamma( x))(f_{\zeta}, \psi_\zeta( x) ) = ( f_\gamma f_{\zeta},f_{\gamma}^*\psi_\zeta(  \psi_\gamma( x))) ,\]
which is the composite
\[ \xymatrix{  \Fc \ar[r]^-{\psi_\gamma} & f_\gamma^*   \Fc \ar[r]^-{f_\gamma^* \psi_{\zeta}} &  f_\gamma^* f_{\zeta}^* \Fc  }. \]
 By construction, $f_{\zeta} f_{\gamma} = f_{\gamma}f_{\zeta}$.  Since $\psi_{\gamma}$ is defined over $\Rn$,  $f_{\zeta}^*\psi_\gamma=\psi_\gamma$. Also, $f_\gamma^*\psi_\zeta = \psi_\zeta$ since $\psi_\zeta (x)=x$.
Finally, $ \psi_\gamma(  \psi_\zeta( x))  = \psi_\zeta(  \psi_\gamma( x)) $ since $ \psi_\zeta(x)=x$ is the identity for composition of power series.
\end{proof}

\subsection{Action of the Galois group}
The Galois group $\Gal=\Gal(k/\F_2)$ acts on $W(k)$, and this gives an action of $\Gal$ on $\Rnper(k)$ by acting on the coefficients.
Let $\sigma\in \Gal$ be the Frobenius. The action of $\sigma$ on $\Rnper(k)$ is a ring isomorphism which we denote by $f_{\sigma}$. Note that
\[ \Rnper(k)^{\Gal} =\Z_2\otimes_{\Z_{(2)}}\Rnper = \Rnper(\F_2) . \]
Since $\Fc$ is defined over $\Rn$, $f_{\sigma}^*\Fc =\Fc$. So, letting $\psi_{\sigma}(x)=x$, we get an action of $\Gal$ on $(\Rnper(k),u, \Fc )$ via  the morphism
\[(f_\sigma, \psi_{\sigma}) \colon \Fc \to \sigma^*\Fc .\]
The group $\Gal$ acts on $k^{\times}$ via its action on $k$.
This extends the actions of $\Gal$ and $k^{\times}$ to an action of $\Gal \ltimes k^{\times}$ on $(\Rnper(k),u , \Fc)$. Both the $\Gal$-action and the $k^{\times}$-action commute with the $C_{2^n}$-action, so we get an action of
\[  \Gk:=C_{2^n} \times \Ckm\]
on $(\Rnper(k),u,\Fc)$, where $\Ckm = \Gal\ltimes k^{\times}$.

We now have the following result, which combines all of these actions:
\begin{proposition}\label{prop:keyhomomorphism}
Let $R$ be an even periodic graded ring with an action of $\Gk$
which restricts to the involution on $C_2$.
A $\Gk$-equivariant ring homomorphism
\[ f \colon \Rnper(k) \to R \]
determines a formal group law $(R, f(u), f^*\Fc)$ over $R$ with an action of $\Gk$ by maps of $2$-typical formal group laws that extends the $C_2$-conjugation action.
\end{proposition}

In the next sections, we will be considering quotients of the ring $\Rn$, and the action of $\Gk$ does not descend to these quotients. However, the action of certain subgroups of $\Gk$ does.

Let $m \geq 1$ and $q=2^m-1$. Let $k^{\times}[q] \subseteq k^{\times}$ be the $q$-torsion, that is, the subgroup of elements $\zeta\in k^{\times}$ so that $\zeta^{q}=1$.
Let  $\Gkm \subseteq \Gk$ be the subgroup
\begin{align}\label{eq:Gkm}
\Gkm=C_{2^n} \times (\Gal\ltimes k^{\times}[q]).\end{align}
Now, let
\begin{align}\label{eq:Rnkm}
\Rnper(k)\langle m \rangle :=W(k)[C_{2^n} \cdot t_1^{C_{2^n}},\ldots, C_{2^n} \cdot t_m^{C_{2^n}}][C_{2^n} \cdot u^{\pm 1}]/(C_{2^n} \cdot (t_m^{C_{2^n}}-u^{2^m-1})).
\end{align}
\begin{proposition}\label{prop:killm}
The ring $\Rnper(k)\langle m \rangle $ is the quotient of $\Rnper(k) $ by an invariant ideal, and thus inherits an action of $\Gkm$ via the quotient map.
\end{proposition}
\begin{proof}
The ideal $(C_{2^n}\cdot t_{m+1}^{C_{2^n}}, C_{2^n}\cdot t_{m+2}^{C_{2^n}}, \ldots)$ is invariant  by definition.
We also have that
\[f_\zeta(t_m^{C_{2^n}}) =t_m^{C_{2^n}}\]
and
\[f_\zeta(u^{q}) = \zeta^{-q}u^{q}=u^{q}\]
since $\zeta$ is a $q$th root of unity. It follows that the ideal
\[(C_{2^n} \cdot (t_m^{C_{2^n}}-u^{2^m-1})) = ( t_{m}^{C_{2^n}}-u^{q}, \gamma t_{m}^{C_{2^n}}-\gamma u^{q}, \ldots,  \gamma^{2^{n-1}-1} t_{m}^{C_{2^n}}- \gamma^{2^{n-1}-1}  u^{q})\]
is invariant.
\end{proof}

\begin{proposition}\label{cor:finalact}
Let $R$ be an even periodic graded ring with an action of $\Gkm$
which restricts to the involution on $C_2$.
A $\Gkm$-equivariant ring homomorphism
\[ f \colon \Rnper(k)\langle m\rangle \to R \]
gives rise to a formal group law $(R, f(u), f^*\Fc)$ over $R$ with an action of $\Gkm$ by maps of $2$-typical formal group laws that extends the $C_2$-conjugation action.
\end{proposition}

\subsection{Lubin--Tate theories and the Morava stabilizer group}

Consider the subcategory $\mathcal{M}_{FG}^{h,\per,c}$ of $\mathcal{M}_{FG}^{h,\per}$  whose objects are triples $(R,u, F)$ with  $R$  a complete local ring.  Morphisms in this subcategory are pairs $(f,\psi)$ where $f \colon R \to S$ is a continuous morphism, so that the image under $f$ of the maximal ideal of $R$ is contained in the maximal ideal of $S$.

Let $\mathcal{M}_{FG}$ be as in \cite[Section 11]{HHR}. Its objects are non-graded formal group laws over non-graded rings. Morphisms in $\mathcal{M}_{FG}$ are pairs $(f,\psi) : (R,F) \to (S,G)$ where $\psi \colon G \to f^*F$ is any non-graded isomorphism. Let $\mathcal{M}_{FG}^c$ be the subcategory of $\mathcal{M}_{FG}$, where we restrict as above to complete local rings and continuous homomorphisms.

Let $K = k[\bar u^{\pm 1}]$ be an even periodic graded ring with the property that $k$ is a finite field of characteristic $2$ and $\bar u$ has degree $2$. Then $\Gal(k/\F_2)$ acts on $K$ via its action on $k$.
Let $ \Gamma_h$ be a homogenous formal group law of height $h=2^{n-1}m$ defined over $K$. Then $(K,\bar u, \Gamma_h)\in \mathcal{M}_{FG}^{h,\per}$.

\begin{definition}\label{defn:moravastab}
The (big) Morava stabilizer group $\G(k, \Gamma_h)$ is the group of automorphisms of $(K,\bar u, \Gamma_h)\in\mathcal{M}_{FG}^{h,\per}$. Let $\mathbb{S}(k,\Gamma_h)$ be the subgroup whose elements are those the pairs of the form $(\id, \psi)$.
\end{definition}
\begin{remark}
Suppose that $\Gamma_h$ is defined over $K^{\Gal} \cong \F_2[\bar u^{\pm 1}]$. Then
there is a split exact sequence
\[\xymatrix{1 \ar[r] &\mathbb{S}(k,\Gamma_h) \ar[r] &\G(k, \Gamma_h) \ar[r] & \Gal(k/\F_2) \ar[r] & 1 \ .}  \]
\end{remark}

\begin{remark}
Using the equivalence $\Phi$ of \eqref{eq:equicats},
$\G(k, \Gamma_h)$ is the group of automorphisms of
\[(K,\bar u,  \wt \Gamma_h)\in\mathcal{M}_{FG}^{\per}\]
where
$ \wt \Gamma_h $ is as in Notation~\ref{not:ungradedfgl}.
This, in turn, is isomorphic to the group of automorphisms of the pair $(k,  \wt \Gamma_h) \in \mathcal{M}_{FG}$. So,  $\G(k, \Gamma_h)$ as defined above is just the usual (big) Morava stabilizer group. \end{remark}

\begin{remark}
Note that any automorphism of $K$ is continuous so  $\G(k, \Gamma_h)$  is also  the group of automorphisms of $(K,\bar u, \Gamma_h)\in\mathcal{M}_{FG}^{h,\per,c}$.
\end{remark}

Let $\Rkm$ be an even periodic complete local ring with maximal ideal $\mathfrak{m}$ and $u\in \Rkm_2$ a choice of unit. Let $p \colon \Rkm\to \Rkm/\mathfrak{m}$ be the quotient map and
\[\iota \colon K \xrightarrow{\cong} \Rkm/\mathfrak{m} \]
be an isomorphism such that $\iota(\bar u) = p(u)$. Let $F_h$ be a $2$-typical homogenous  formal group law over $R$.
Suppose further that
\[\iota^*\Gamma_h =p^* F_h\]
and that $(\Rkm_0, \wt F_h)$ is a universal deformation of $(k, \wt \Gamma_h)$ in the sense of Lubin and Tate \cite{lubintate}. The Lubin--Tate theorem implies that there are isomorphisms
\begin{align}\label{eq:morcont}
 \G(k, \Gamma_h) & \cong \Aut_{\mathcal{M}_{FG}^c}( (\Rkm_0, \wt F_h))
 \cong \Aut_{\mathcal{M}_{FG}^{h, \per,c}}( (\Rkm, u,  F_h)).\end{align}

\begin{remark}
Note that the data $(\Rkm, u, F_h)$ together with $\iota$ is not unique. However, it is unique up to unique $\star$-isomorphism. That is, for $i = 1, 2$, given two choices $(\Rkm_i, u_i, F_{h,i})$ with
\[\iota_i \colon K\xrightarrow{\cong} \Rkm_i/\mathfrak{m}_i \]
and $p_i\colon  \Rkm_i \to \Rkm_i/\mathfrak{m}_i$, there exists a unique isomorphism
\[ (f, \psi )  \colon (\Rkm_1, u_1, F_{h,1}) \to  (\Rkm_2, u_2, F_{h,2})   \]
with the following properties.
\begin{itemize}
\item $f$ is continuous, and so induces an isomorphism $\bar f$ on residue fields.
\item $\bar f \circ \iota_1 = \iota_2$
\item $p_2^*\psi \colon p_2^*F_{h,2} \to p_2^* f^*F_{h,1}$ is the identity on $\iota_2^*\Gamma_h$. This makes sense since both $p_2^*F_{h,2} = \iota_2^*\Gamma_h$ and
\[p_2^* f^*F_{h,1} = \bar f^*p_1^*F_{h,1}  = \bar f^*\iota_1^*\Gamma_h = \iota_2^*\Gamma_h.\]
\end{itemize}
\end{remark}


\section{Equivariant deformation parameters and recursion formulas}\label{sec:formulas}
Recall from Section~\ref{sec:intro} that
$$\Rn := \pi_*^e \BPCn = \mathbb{Z}_{(2)}[C_{2^n} \cdot t_1^{C_{2^n}}, C_{2^n} \cdot t_2^{C_{2^n}}, \ldots].$$
for all $n \geq 1$.  Under the inclusion map $\Rr \hookrightarrow \Rn$,
we can view the generators $C_{2^r} \cdot t_k^{C_{2^r}}$ of $\Rr$ as elements of $\Rn$ for every $1 \leq r \leq n$.

The goal of this section is to prove Theorem ~\ref{theorem:Mform}, which gives a recursive formula relating the $t_k^{C_{2^{r}}}$-generators for various values of $r$.  This formula will be essential for proving Theorem~\ref{thm:model}.  Recall that the ideals $I_k \subset \Rn$ are defined as
$$I_k := (2,v_1,\ldots, v_{k-1}),$$
where the elements $v_i \in \pi_*BP = \Rone \subset \Rn$ are the Araki generators, so the coefficients of the 2-series of the formal group law $\Fc$.

As stated in Theorem~\ref{theorem:Mform}, we prove that for every $n\geq 2$ and $k \geq 1$,
\begin{eqnarray*}
t_k^{C_{2^{n-1}}}
&\equiv& t_k^{C_{2^n}} + \gamma_n t_k^{C_{2^n}} + \sum_{j=1}^{k-1}\gamma_{n}{t}_{j}^{C_{2^n}} ({t}_{k-j}^{C_{2^n}})^{2^j} \pmod{I_k}.
\end{eqnarray*}
In this section, we will also prove the equality
$$v_k \equiv t_k^{C_2} \pmod{I_{k}}$$
for all $k \geq 1$.  Once we have established this, Equation~(\ref{eq:tkrecursionformula}) will give a formula for the $v_k$-generators in terms of the ${t}_{j}^{C_{2^n}}$-generators for all $n \geq 1$.

\subsection{The logarithms of $\Fc$ and $\Fc^{\mathrm{add}}$}
To prove Theorem~\ref{theorem:Mform}, we begin by studying the relationship between the logarithms of the formal group laws $\Fc$ and $\Fc^{\gamma_n}$.  
Since $\Rn$ is $2$-torsion free, the formal group law $\Fc$ admits a logarithm over $\Rn \otimes \mathbb{Q}$. This is an isomorphism
$$\begin{tikzcd} \log_{\Fc} \colon \Fc \ar[r, "\cong"]&  \Fc^{\mathrm{add}}, \end{tikzcd}$$
where $\Fc^{\mathrm{add}}(x,y)=x+y$ is the additive formal group law.  Define $\ell_i \in \Rn \otimes \Q$ to be the coefficients of the logarithm of $\Fc$:
\[\log_{\Fc} (x) =\sum_{i\geq 0} \ell_i x^{2^i},\]
and the elements $v_i \in \Rn$ to be the coefficients of the 2-series of $\Fc$:
\[{[}2{]}_{\Fc}(x) =  \sum_{i\geq 0}{}^{\Fc} v_i x^{2^i}.\]
These are the \emph{Araki generators}, \cite{Araki1973}. We provide the proof for the following standard result which will be crucial below. See also \cite[A2.2]{ravgreen}.
\begin{proposition}\label{prop:elldenom}
For every $k\geq 1$,
\begin{eqnarray}
2\ell_k &=& 2^{2^k}\ell_k+ \sum_{j=1}^{k-1} \ell_{k-j} v_{j}^{2^{k-j}} +v_k. \label{eqn:lkvk}
\end{eqnarray}
Furthermore,  there exists $x_k\in \Rn$ so that
\[ \ell_k  = 2^{-k} x_k\]
and $x_k \neq 0 $ modulo $2$. That is, the denominator of $\ell_k$ in $\Rn \otimes \Q$ is exactly $2^k$.
\end{proposition}
\begin{proof}
Since $\log_{\Fc}$ is a homomorphism of formal group laws,
\[ \log_{\Fc}([2]_{\Fc}( x)) = [2]_{\Fc^{\mathrm{add}}}(\log_{\Fc}(x)) = 2\log_{\Fc}(x).\]
This implies that
\[ \sum_{k\geq 0}2 \ell_k x^{2^k}  = \log_{\Fc}\left( \sum_{i\geq 0}{}^{\Fc} v_i x^{2^i} \right)  = \sum_{i,j\geq 0}  \ell_j v_i^{2^j} x^{2^{i+j}} . \]
For $k\geq 1$, comparing the coefficients of $x^{2^k}$ on both sides of the equation above produces the equality
\begin{eqnarray*}
2\ell_k &=& 2^{2^k}\ell_k+ \sum_{j=1}^{k} \ell_{k-j} v_{j}^{2^{k-j}} \\
&=& 2^{2^k}\ell_k+ \sum_{j=1}^{k-1} \ell_{k-j} v_{j}^{2^{k-j}} + v_k.
\end{eqnarray*}
This proves the first claim.

To prove that the denominator of $\ell_k$ in $\Rn \otimes \mathbb{Q}$ is exactly $2^k$, we use induction on $k$.  For the base case, when $k =1$, the equation above gives the equality
\[\ell_1 = \frac{v_1}{2 -2^2} = \frac{-v_1}{2}.\]
For the induction step, suppose that for all $1\leq i\leq k-1$, we have $\ell_{i} = 2^{-i}x_i$ where $x_i \in \Rn$ and $x_i \neq 0 \pmod{2}$. Then solving for $\ell_k$ in the formula above gives
\begin{align*}
\ell_k &= \frac{1}{2-2^{2^k}}\left(\sum_{j=1}^{k}  2^{-(k-j)} x_{k-j} v_{j}^{2^{k-j}} \right)\\
& = 2^{-k} \frac{1}{1-2^{2^k-1}} \left( \sum_{j=1}^{k}  2^{j-1} x_{k-j} v_{j}^{2^{k-j}}\right).
\end{align*}
Let $x_k =\frac{1}{1-2^{2^k-1}} \left( \sum_{j=1}^{k}  2^{j-1} x_{k-j} v_{j}^{2^{k-j}}\right)$. Then, $x_k \in \Rn$ and $x_k = x_{k-1}v_1^{2^{k-1}}$ modulo $2$, which, by the induction hypothesis, is non-zero. This proves the second claim.
\end{proof}

\subsection{The $\Rn$-modules $L_k$ and the ideals $I_k$}
\begin{definition}\label{defn:Mk}
Let $L_k$ be the $\Rn$-submodule of $\Rn \otimes \Q$ generated by the elements $\{2, \ell_1, \ldots, \ell_{k-1}\}$.  More specifically, a generic element of $L_k$ has the form
$$r_0 \cdot 2 + r_1 \cdot \ell_1 + r_2 \cdot \ell_2 + \cdots + r_{k-1} \cdot \ell_{k-1},$$
where $r_i \in \Rn$ for all $0 \leq i \leq k-1$.

More generally, for $1 \leq r\leq n$ and $0 \leq j \leq 2^{r}-1$, let ${\gamma_r^j}L_k$ be the $\Rn$-submodule of $\Rn\otimes \Q$ generated by the elements $\{2, \gamma_r^j\ell_1,\ldots, \gamma_r^j\ell_{k-1}\}$.  A generic element of ${\gamma_r^j}L_k$ has the from
\[r_0 \cdot 2 + r_1 \cdot \gamma_r^j\ell_1 +\ldots + r_{k-1} \cdot \gamma_r^j\ell_{k-1},\]
where $r_i \in \Rn$ for all $0 \leq i \leq k-1$.
\end{definition}
We will first deduce an analogue of Theorem~\ref{theorem:Mform}, but modulo $\gamma_{n-1}L_{k}$ rather than $I_k$.
We will also establish that $\Rn \cap {\gamma_r} L_k =I_k$ for all $n$, $k$, and $1 \leq r\leq n$, thus the relevance of the modules  $\gamma_{r}L_{k}$. The first step for this is the following proposition.

\begin{proposition}\label{prop:danny1alt}\label{prop:LkcapR=Ik}
For every $k\geq 1$,
\[L_{k} \cap \Rn= I_k. \]
\end{proposition}
\begin{proof}
We prove the claim by using induction on $k$.  The base case when $k = 1$ is clear: an element in $L_1$ is of the form $r_0 \cdot 2$, where $r_0 \in \Rn$.  Therefore,
\[L_1= L_1 \cap \Rn = (2) = I_1.\]

Now, suppose that $L_{k-1} \cap \Rn = I_{k-1}$.  Let
\[ x = r_0 \cdot 2 + r_1 \cdot \ell_1 + \cdots  + r_{k-2} \cdot\ell_{k-2} + r_{k-1} \cdot\ell_{k-1}\]
be an element in $L_k \cap \Rn$.  By Proposition~\ref{prop:elldenom}, the denominator of $\ell_{i}$ is exactly $2^{i}$ for all $0 \leq i \leq k-1$.  Since $x\in \Rn$, we must have that $r_{k-1} = 2 r_{k-1}'$ for some $r_{k-1}' \in \Rn$.  We can rewrite $x$ as
\begin{align*}
x &= r_0 \cdot 2 + r_1 \cdot \ell_1 + \cdots + r_{k-2} \cdot \ell_{k-2} + 2 r_{k-1}' \cdot \ell_{k-1} \\
&= r_0 \cdot 2 + r_1 \cdot \ell_1 + \cdots + r_{k-2} \cdot \ell_{k-2} + r_{k-1}'\cdot (2\ell_{k-1}).
\end{align*}
Now, note that $r_0 \cdot 2 + r_1 \cdot \ell_1 + \cdots + r_{k-2} \cdot \ell_{k-2}$ is in $L_{k-1}$.  Furthermore, Equation~(\ref{eqn:lkvk}) implies that
\begin{align*}
2\ell_{k-1}  \equiv v_{k-1} \pmod{L_{k-1}}.
\end{align*}
Therefore,
\[x-r_{k-1}' \cdot v_{k-1} = \ell \in L_{k-1}. \]
Since the left hand side is in $\Rn$, so is the right hand side.  By the induction hypothesis, $L_{k-1} \cap \Rn = I_{k-1}$.  Therefore
$$x-v_{k-1} r_{k-1} \in I_{k-1}$$
and so $x\in I_{k}$.  This proves that $L_{k} \cap \Rn \subseteq I_k$. The other inclusion is an immediate consequence of  Proposition~\ref{prop:elldenom}.
\end{proof}

\subsection{Comparison of generators}
Next, we establish a formula which relates the $t_k^{C_{2^n}}$-generators and the $\ell_k$-generators. To do this, note that the logarithm  $\log_{\Fc^{\gamma_n}} \colon \Fc^{\gamma_n} \to \Fc^{\mathrm{add}}$ is given by
\begin{align*}
\log_{\Fc^{\gamma_n}}(x) &=\sum_{i\geq 0} (\gamma_{n}\ell_i) x^{2^i}.
\end{align*}
Furthermore, there is a commutative diagram
\begin{equation}\label{eq:logdiag}
\xymatrix{\Fc \ar[rr]^-{\psi_{\gamma_{n}}} \ar[dr]_-{\log_{\Fc}} & & \Fc^{\gamma_{n}} \ar[dl]^-{\log_{\Fc^{\gamma_{n}}}} \\
&\Fc^{\mathrm{add}} & } \end{equation}
From this, we deduce the following key formulas which will be used throughout this section.

\begin{proposition}\label{prop:keyformula}
For every $n\geq 1$,
\begin{eqnarray}
{\ell}_k - \gamma_{n}{\ell}_k  &=&    \sum_{j=0}^{k-1}\gamma_{n}{\ell}_j  ({t}_{k-j}^{C_{2^n}})^{2^j} \label{eq:3.5eq1},\\
 \gamma_{n}{\ell}_k- \gamma_{n-1}{\ell}_k  &=& \sum_{j=0}^{k-1}\gamma_{n-1}{\ell}_j  (\gamma_{n}{t}_{k-j}^{C_{2^n}})^{2^j} \label{eq:3.5eq2},\\
   {\ell}_k- \gamma_{{n-1}} {\ell}_k  &=& \sum_{j=0}^{k-1}\left( \gamma_{n}{\ell}_j  ({t}_{k-j}^{C_{2^n}})^{2^j} + \gamma_{{n-1}}{\ell}_j  (\gamma_{n}{t}_{k-j}^{C_{2^n}})^{2^j} \right). \label{eq:3.5eq3}
\end{eqnarray}
\end{proposition}

\begin{proof}
 The commutativity of the diagram \eqref{eq:logdiag} implies
\begin{align*}
 \sum_{i\geq 0} {\ell}_i {x}^{2^i}  &= \log_{\Fc} (x  )  \\
  &= \log_{\Fc^{\gamma_n}}(\psi_{\gamma_n}(x)  )  \\
 &= \log_{\Fc^{\gamma_n}}\left( \sum_{i\geq 0}{}^{\Fc^{\gamma_n}} {t}_i^{C_{2^n}} x^{2^i} \right)  \\
 &=\sum_{i\geq 0}  \log_{\Fc^{\gamma_n}}({t}_i^{C_{2^n}}  x^{2^i} ) \\
 &=\sum_{i,j \geq 0}   \gamma_n{\ell}_j   ({t}_i^{C_{2^n}})^{2^j} x^{2^{i+j}}.\\
\end{align*}
Comparing the coefficients for $x^{2^k}$ gives the equation
\begin{eqnarray*}
 {\ell}_k &=& \sum_{i+j=k} \gamma_n{\ell}_j  ({t}_i^{C_{2^n}})^{2^j}  \\
 &=& \sum_{j=0}^k \gamma_n{\ell}_j  ({t}_{k-j}^{C_{2^n}})^{2^j} \\
 &=& \gamma_n{\ell}_k +   \sum_{j=0}^{k-1}\gamma_n{\ell}_j  ({t}_{k-j}^{C_{2^n}})^{2^j}.
\end{eqnarray*}
This proves the first equation.  Applying $\gamma_{n}$ to the first equation and using the fact that $\gamma_{n}^2 =  \gamma_{{n-1}}$ proves the second equation.  Adding the first two equations together proves the third equation.
\end{proof}

We can now relate the elements $t_k^{C_2} $ to the coefficients $v_k$ of the $2$-series of $\Fc$.
\begin{proposition}\label{prop:tkvk}
For all $k \geq 1$, in $\Rn$, we have the equality
\[t_k^{C_2} \equiv v_k \pmod{I_{k}}.\]
\end{proposition}
\begin{proof}
Letting $n=1$ in Equation~\eqref{eq:3.5eq1} gives the equation
$$t_k^{C_2} = 2 \ell_k - \sum_{j = 1}^{k-1}\ell_j(t_{k-j}^{C_2})^{2^j}.$$
Multiplying this equation by the unit $(1-2^{2^k-1})$ yields the formula
\begin{equation*}
(1-2^{2^k-1})t_k^{C_2} = (2-2^{2^k}) \ell_k -(1-2^{2^k-1}) \sum_{j =1}^{k-1} \ell_j (t_{k-j}^{C_{2}})^{2^j}
\end{equation*}
From Proposition~\ref{prop:elldenom}, we also have the formula
$$v_k = (2 - 2^{2^k})\ell_k - \sum_{j =1}^{k-1} \ell_{k-j}v_j^{2^{k-j}} = (2 - 2^{2^k})\ell_k - \sum_{j =1}^{k-1} \ell_j v_{k-j}^{2^j}.$$
Subtracting these two formulas and rearranging terms gives
\begin{eqnarray*}
 t_k^{C_2} - v_k &=&-(1-2^{2^k-1})  \sum_{j =1}^{k-1} \ell_j (t_{k-j}^{C_{2}})^{2^j} + \sum_{j = 1}^{k-1} \ell_j v_{k-j}^{2^j} +2^{2^k-1} t_k^{C_2}
\end{eqnarray*}
The right hand side is in $L_{k}$. Since $ t_k^{C_2} - v_k $ is also in $\Rn$,
\[t_k^{C_2} - v_k \in L_k \cap \Rn = I_{k}\]
by Proposition~\ref{prop:danny1alt}.  It follows that $ t_k^{C_2} \equiv v_k \pmod{I_k}$.
\end{proof}

In the following results, we establish the relationship between the $\Rn$-submodules $\gamma_{r}L_{k}$ of $\Rn \otimes \Q$ and the ideals $I_k = (2,v_1, \ldots, v_{k-1})$ of $\Rn$.

\begin{proposition}\label{prop:danny1}\label{prop:reducelk-glk}
For all $n,k\geq 1$,
\[ \ell_k -\gamma_n \ell_k  \in  \Rn + L_k.\]
\end{proposition}

\begin{proof}
We will fix $n$ and abbreviate notations by writing $\Rc=\Rn$, $\gamma = \gamma_n$ and $t_i = t_i^{C_{2^n}}$.  We will prove the claim by using induction on $k$.  The base case when $k =1$ is immediate because by Equation~\eqref{eq:3.5eq1},
\[\ell_1 - \gamma \ell_1 =  t_1 \in \Rc.\]
Now, suppose that we have proven that
\[ \ell_{i} -\gamma \ell_{i} \in \Rc + L_{i}\]
for all $1 \leq i \leq k-1$.  Equation~\eqref{eq:3.5eq1} shows that
\begin{eqnarray*}
\ell_k - \gamma \ell_k &=& \sum_{j= 0}^{k-1} \gamma \ell_j t_{k-j}^{2^j} \\
&=&\sum_{j = 0}^{k-1}(\gamma \ell_j - \ell_j) t_{k-j}^{2^j} + \sum_{j=0}^{k-1}\ell_j t_{k-j}^{2^j}.
\end{eqnarray*}
By our induction hypothesis, every term in the first sum is an element in $\Rc + L_k$.  The second sum is also in $\Rc + L_k$.  Therefore,
\[ \ell_k - \gamma \ell_k \in \Rc+L_k. \qedhere\]
This completes the inductive step.
\end{proof}

\begin{proposition}\label{prop:IkInvariantIdeal}\label{prop:danny2}
For all $n, k \geq 1$,
\[  v_k-\gamma_n v_k \in I_k.\]
In other words, the ideals $I_k \subseteq \Rn$ are invariant under the $C_{2^n}$-action.
\end{proposition}
\begin{proof}
We will again fix a specific $n$ and write $\Rc=\Rn$, $\gamma = \gamma_n$ and $t_i = t_i^{C_{2^n}}$. First note that by Proposition~\ref{prop:elldenom} and Equation~\ref{eq:3.5eq1},
\[ v_1 - \gamma v_1 = -2(\ell_1 - \gamma \ell_1) =-2t_1  \in I_1.\]
So the claim holds when $k=1$.

Assume that the claim holds for $v_1, \ldots,v_{k-1}$. We will show that $v_k - \gamma v_k \in L_k$. The fact that $v_k - \gamma v_k \in \Rc$ will imply that it is actually in $L_k \cap \Rc$, which is equal to $I_k$ by Proposition~\ref{prop:danny1alt}.

We use Propositions~\ref{prop:elldenom} to make the following computation:
\begin{align*}
v_k - \gamma v_k& = \left((2-2^{2^k}) \ell_k - \sum_{j =1}^{k-1} \ell_j v_{k-j}^{2^j}\right) -\left( (2 -2^{k}) \gamma \ell_k - \sum_{j =1}^{k-1} \gamma \ell_j \gamma v_{k-j}^{2^j}\right) \\
&\equiv (2-2^{2^k})(\ell_k - \gamma \ell_k) - \sum_{j =1}^{k-1} \gamma \ell_j \gamma v_{k-j}^{2^j}  \pmod{L_k} \\
&\equiv (2-2^{2^k})(\ell_k - \gamma \ell_k) + \sum_{j =1}^{k-1} ( \ell_j-\gamma \ell_j) \gamma v_{k-j}^{2^j} \pmod{L_k}.
\end{align*}
By Proposition~\ref{prop:reducelk-glk},  for all $1 \leq j \leq k$, $\ell_j - \gamma \ell_j = r_j $ modulo $L_j \subseteq L_k$ for some $r_j \in \Rc$.  Using the inductive hypothesis that
\[\gamma v_{i} \equiv v_i  \pmod {{I_i}\subseteq L_k},\]
the expression above can be further reduced modulo $L_k$ to show that
\begin{eqnarray*}
v_k - \gamma v_k \equiv (2-2^{2^k})r_k + \sum_{j =1}^{k-1} r_j v_{k-j}^{2^j}  \pmod{L_k}.
\end{eqnarray*}
Since the right-hand side is in $ I_k \subseteq {L_k}$, this shows that $v_k - \gamma v_k  \in L_k$.  It follows that $v_k - \gamma v_k \in L_k \cap \Rc = I_k$.
\end{proof}

\begin{cor}\label{cor:gammaLkisIk}\label{cor:MkcapRn}
For every $n, k \geq 1$ and $1 \leq r \leq n$,
\[{\gamma_r} L_k\cap \Rn= I_k.\]
\end{cor}
\begin{proof}
Proposition~\ref{prop:LkcapR=Ik} and Proposition~\ref{prop:IkInvariantIdeal} imply that
\[{\gamma_r}L_k \cap \Rn = \gamma_r I_k = I_k = L_k \cap \Rn. \qedhere\]
\end{proof}

\begin{proposition}\label{prop:Mform}
For every $n\geq 2$,
\[ t_k^{C_{2^{n-1}}} \equiv \sum_{j=0}^{k}\gamma_{n}{t}_{j}^{C_{2^n}} ({t}_{k-j}^{C_{2^n}})^{2^j} \pmod{\gamma_{n-1}L_{k}} \]
\end{proposition}

\begin{proof}
Equation~\eqref{eq:3.5eq1}, with $n$ replaced by $n-1$, states that
\begin{align*}
{\ell}_k - \gamma_{n-1}{\ell}_k  &=    \sum_{j=0}^{k-1}\gamma_{n-1}{\ell}_j  ({t}_{k-j}^{C_{2^{n-1}}})^{2^j}  \\
&\equiv {t}_{k}^{C_{2^{n-1}}}\pmod{\gamma_{n-1}L_{k}}.
\end{align*}

In the following steps, we use the relations in Proposition~\ref{prop:keyformula}  repeatedly:
\begin{align*}
t_k^{C_{2^{n-1}}}  & \equiv {\ell}_k- \gamma_{n-1} {\ell}_k \pmod{\gamma_{n-1}L_{k}} \\
&\equiv \sum_{j=0}^{k-1}\left(\gamma_{n}{\ell}_j  ({t}_{k-j}^{C_{2^n}})^{2^j} + \gamma_{n-1}{\ell}_j  (\gamma_n{t}_{k-j}^{C_{2^n}})^{2^j}\right)\pmod{ \gamma_{n-1}L_{k}} \text{ (by Equation~\eqref{eq:3.5eq3})}\\
&\equiv {t}_{k}^{C_{2^n}} + \gamma_{n}t_k^{C_{2^n}}+  \sum_{j=1}^{k-1}\left(\gamma_n{\ell}_j  ({t}_{k-j}^{C_{2^n}})^{2^j} + \gamma_{n-1}{\ell}_j  (\gamma_n {t}_{k-j}^{C_{2^n}})^{2^j}\right)\pmod{ \gamma_{n-1}L_{k}} \\
&\equiv {t}_{k}^{C_{2^n}} + \gamma_{n}t_k^{C_{2^n}}+  \sum_{j=1}^{k-1}\gamma_n{\ell}_j  ({t}_{k-j}^{C_{2^n}})^{2^j}  \pmod{\gamma_{n-1}L_{k}} \\
&\equiv {t}_{k}^{C_{2^n}} + \gamma_{n}t_k^{C_{2^n}}+  \sum_{j=1}^{k-1}(\gamma_n{\ell}_j  -  \gamma_{n-1}{\ell}_j )({t}_{k-j}^{C_{2^n}})^{2^j}  \pmod{\gamma_{n-1}L_{k}} \\
&\equiv {t}_{k}^{C_{2^n}} + \gamma_{n}t_k^{C_{2^n}}+  \sum_{j=1}^{k-1}\left( \sum_{r=0}^{j-1}\gamma_{n-1}{\ell}_r  (\gamma_n{t}_{j-r}^{C_{2^n}})^{2^r} \right)({t}_{k-j}^{C_{2^n}})^{2^j}  \pmod{\gamma_{n-1}L_{k}} \\
&\hspace{0.2in} \text{(by Equation~\eqref{eq:3.5eq2})}\\
&\equiv {t}_{k}^{C_{2^n}} + \gamma_{n}t_k^{C_{2^n}}+  \sum_{j=1}^{k-1}\gamma_n{t}_{j}^{C_{2^n}} ({t}_{k-j}^{C_{2^n}})^{2^j}  \pmod{\gamma_{n-1}L_{k}}.
\end{align*}
\end{proof}

\begin{proof}[Proof of Theorem~\ref{theorem:Mform}.]
By Proposition~\ref{prop:Mform},
\[ t_k^{C_{2^{n-1}}} - \left({t}_{k}^{C_{2^n}} + \gamma_{n}t_k^{C_{2^n}}+  \sum_{j=1}^{k-1}\gamma_n{t}_{j}^{C_{2^n}} ({t}_{k-j}^{C_{2^n}})^{2^j}\right) \in {\gamma_{n-1}}L_k.\]
However, all terms are in $\Rn$.  Therefore, this difference is in $\gamma_{n-1}L_k \cap \Rn = I_k$ by Corollary~\ref{cor:MkcapRn}.  The result follows.
\end{proof}


\section{Deformations with group actions}\label{sec:proofs-alg}

In this section,
we construct a formal group law $\Gamma_h$ of height $h=2^{n-1}m$, together with a universal deformation $F_h$ of $\Gamma_h$. The formal group law $\Gamma_h$ comes with an obvious action of
\[\Gkm = C_{2^n} \times (\Gal \ltimes k^{\times}[q])\]
and we study its universal deformation together with its action of $\Gkm$.

\subsection{The formal group law $\Gamma_h$}

Let $h=2^{n-1}m$.  Consider the functor from the category of finite fields of characteristic $2$ to the category of complete local rings which takes $k$ to the ring
\begin{align*}
\Rkm =W(k) [C_{2^n} \cdot t_1^{C_{2^n}}, \ldots, C_{2^n} \cdot t_{m-1}^{C_{2^n}}, C_{2^{n}} \cdot u ] [C_{2^{n}}\cdot u^{-1} ]^{\wedge}_\mathfrak{m},
\end{align*}
where
\[ \mathfrak{m}=(C_{2^n}\cdot t_1^{C_{2^n}}, \ldots, C_{2^n}\cdot t_{m-1}^{C_{2^n}}, C_{2^n}\cdot (u-\gamma_n u)).\]
Here, $|t_i^{C_{2^n}}| = 2(2^i-1)$ for $1 \leq i \leq m-1$ and $|u| =2$. Note further that, for a graded ring $A$ with graded ideal $I$, by $A^{\wedge}_{I}$ we mean the graded ring whose $s$th homogenous component is $(A^{\wedge}_I)_s = \varprojlim_i A_s/I_0^iA_s$ for $I_0 = A\cap I$.

There is an action of the group $\Gkm$ on the ring $\Rkm$.  To describe this action, note that there is an action of $C_{2^n}$ on $\Rkm$ by
$W(k)$-linear maps, determined by
\[\gamma_n(\gamma_n^r x) = \left\{ \begin{array}{ll} \gamma_n^{r+1} x & r < 2^{n-1}-1 \\
-x & r = 2^{n-1}-1
\end{array}\right. \]
for $x=t_i^{C_{2^n}}$, $1\leq i\leq m-1$, and $x=u$.  The Galois group $\Gal=\Gal(k/\F_2)$ also acts on $\Rkm$ via its action on the coefficients $W(k)$.  Lastly, the group $k^{\times}[q]$ for $q=2^m-1$ acts on $\Rkm$ by
 \begin{align*}
f_{\zeta}(u)&=\zeta^{-1} u, \\
f_{\zeta}(t_i^{C_{2^n}}) &= t_i^{C_{2^n}}.
\end{align*}
for every $\zeta \in k^{\times}[q]$ and $1 \leq i \leq m-1$.  All together, these three actions combine to give an action of the group $\Gkm$ on $\Rkm$. The ring $\Rkm$ and this action of $\Gkm$ were already discussed in Section~\ref{sec:mainresults}.

\begin{remark}\label{rem:2isin}
Note that $2\in \mathfrak{m}$ as
\begin{eqnarray*}
\gamma_n^{2^{n-1}-1}(u-\gamma_n u) &=&\gamma_n^{2^{n-1}-1}u +u \\
&=& \sum_{r=0}^{2^{n-1}-2}\gamma_n^r(u-\gamma_n u) + 2\gamma_n^{2^{n-1}-1}u
\end{eqnarray*}
and $\gamma_n^{2^{n-1}-1}u $ is a unit.
\end{remark}

From Remark~\ref{rem:2isin}, it is clear that $\Rkm$ is a complete local ring with maximal ideal $\mathfrak m$.  The action of $\Gkm$ is continuous in the topology on $\Rkm$ defined by the maximal ideal $\mathfrak{m}$.

Let $\Rnper(k)\langle m\rangle$ be as in \eqref{eq:Rnkm}.
There is a $\Gkm$-equivariant ring homomorphism
\begin{equation}\label{eq:mapf}
f\colon \Rn \to \Rkm ,
\end{equation}
determined by sending
\begin{align*}
t_i^{C_{2^n}} &\longmapsto \begin{cases}t_i^{C_{2^n}} & 1\leq i \leq m-1, \\
 u^{2^m-1} & i=m, \\
0 & i>m.
\end{cases}
\end{align*}
Note in particular that the map $f$ factors through  $\Rnper(k)\langle m\rangle$.  Let $F_h = f^*\Fc$, where $\Fc$, as before, is the image of the universal formal group law under the inclusion $\Rone \to \Rn$. Let
\[p \colon \Rkm \to \Rkm/\mathfrak{m} =:K\]
be the projection,  where $K=k[\bar u^{\pm1}]$ and $\bar u=p(u)$.
Define
\begin{align}\label{eq:fglgammah}\Gamma_h := p^*F_h.\end{align}
By Proposition~\ref{cor:finalact}, $(\Rkm, u, F_h)$ and $(K, \bar{u}, \Gamma_h)$ are formal group laws with $\Gkm$-actions that extend the $C_2$-conjugation action.

\subsection{Universal deformation of $\Gamma_h$}
For $i \geq 1$ and $1 \leq r \leq n$, we will also denote $v_i$ and $t_i^{C_{2^r}}$ for their images in $\Rkm$ under the map $f$.  Let $I_h \subset \Rkm$ be the ideal
\[I_h = (2, v_1, \ldots, v_{h-1}).\]
For $1 \leq r \leq n$, let $I_{C_{2^r}}$ denote the ideal
$$I_{C_{2^r}} = (2, C_{2^r} \cdot t_1^{C_{2^r}}, \ldots, C_{2^r} \cdot t_{2^{n-r}m-1}^{C_{2^r}}, C_{2^r} \cdot (t_{2^{n-r}m}^{C_{2^r}} - \gamma_r t_{2^{n-r}m}^{C_{2^r}})).$$
More explicitly, we have
\begin{eqnarray*}
I_{C_2} &=& (2, t_1^{C_2}, \ldots, t_{2^{n-1}m-1}^{C_2}, 2 t_{2^{n-1}m}^{C_2})  \\
I_{C_4} &=& (2, C_4 \cdot t_1^{C_4}, \ldots, C_4 \cdot t_{2^{n-2}m-1}^{C_4}, C_4 \cdot (t_{2^{n-2}m}^{C_4} - \gamma_2t_{2^{n-2}m}^{C_4})) \\
&\vdots& \\
I_{C_{2^{n-1}}} &=& (2, C_{2^{n-1}} \cdot t_1^{C_{2^{n-1}}}, \ldots, C_{2^{n-1}} \cdot t_{2m-1}^{C_{2^{n-1}}}, C_{2^{n-1}} \cdot (t_{2m}^{C_{2^{n-1}}} - \gamma_{n-1} t_{2m}^{C_{2^{n-1}}})) \\
I_{C_{2^{n}}} &=& (2, C_{2^{n}} \cdot t_1^{C_{2^{n}}}, \ldots, C_{2^{n}} \cdot t_{m-1}^{C_{2^{n}}}, C_{2^{n}} \cdot (t_{m}^{C_{2^{n}}} - \gamma_{n} t_{m}^{C_{2^{n}}}))
\end{eqnarray*}

By Proposition~\ref{prop:tkvk}, we have the equality
$$I_{C_2} = I_h.$$
In the next results, we prove that $\mathfrak{m} = I_{C_{2^n}} = I_{C_2}$. It is clear from Theorem~\ref{theorem:Mform} that we have the following chain of inclusions:
$$I_{C_2} \subset I_{C_4} \subset \cdots \subset I_{C_{2^n}}.$$

\begin{proposition}\label{prop:mandC2n}
The ideals $\mathfrak{m}$ and $I_{C_{2^n}}$  are equal in $\Rkm$.
\end{proposition}
\begin{proof}
Since $t_m^{C_{2^n}} = u^{2^m-1}$, we have the equality
\begin{align*}
t_m^{C_{2^n}} - \gamma_n t_m^{C_{2^n}} =& u^{2^m-1} - (\gamma_n u)^{2^m-1} \\
=& (u - \gamma_n u) \cdot \sum_{i=0}^{2^m-2} u^{i} (\gamma_n u)^{2^m-2-i} \\
=& (u- \gamma_n u) \cdot \text{unit}.
\end{align*}
Since $2\in \mathfrak{m}$ by Remark~\ref{rem:2isin}, this proves the claim.
\end{proof}

\begin{proposition}\label{prop:invertHHRelementspreprop}
For $1 \leq i \leq n$, the images of the elements $C_{2^i} \cdot {t}_{2^{n-i}m}^{C_{2^i}} \in \Rc_i \subset \Rn$ are invertible in $\Rkm$ and the images of the elements $C_{2^i} \cdot {t}_{k}^{C_{2^i}}$ for $2^{n-i}m<k\leq h=2^{n-1}m$ are zero modulo $I_{h}$.
\end{proposition}

\begin{proof}
We will use downward induction on $i$.  The base case, when $i = n$, is immediate because the elements $C_{2^n} \cdot {t}_m^{C_{2^n}}$ are invertible in $\Rkm$ and the elements $C_{2^n} \cdot {t}_k^{C_{2^n}}$ are identically zero for $k>m$.

Now, suppose we have proven the claim for $i$, where $1 < i \leq n$.  More specifically, suppose we have proven that the images of the elements $C_{2^i} \cdot {t}_{2^{n-i}m}^{C_{2^i}}$ are invertible in $\Rkm$ and the images of the elements $C_{2^i} \cdot {t}_{k}^{C_{2^i}}$ are zero modulo $I_h$ for ${2^{n-i}m< k\leq h}$.
Theorem~\ref{theorem:Mform} and the fact that $2^{n-(i-1)}m\leq h$ implies that
\begin{eqnarray*}
{t}_{2^{n-(i-1)}m}^{C_{2^{i-1}}} &\equiv& {t}_{2^{n-(i-1)}m}^{C_{2^i}} + \gamma_i {t}_{2^{n-(i-1)}m}^{C_{2^i}} + \sum_{j = 1}^{2^{n-(i-1)}m -1} \gamma_i {t}_j^{C_{2^i}} ({t}_{2^{n-(i-1)}m - j}^{C_{2^i}})^{2^j}  \pmod{I_h} \\
&\equiv& \gamma_i {t}_{2^{n-i}m}^{C_{2^i}} \cdot ({t}_{2^{n-i}m}^{C_{2^i}})^{2^{2^{n-i}m}}  \pmod{I_h}.
\end{eqnarray*}
This is because every other term in the first line of the equation has a factor in the set $C_{2^i} \cdot t_k^{C_{2^i}}$, $2^{n-i}m<k\leq h$, which is zero modulo $I_h$ by the induction hypothesis.

Since $\gamma_i {t}_{2^{n-i}m}^{C_{2^i}} \cdot ({t}_{2^{n-i}m}^{C_{2^i}})^{2^{2^{n-i}m}}$ is invertible in $\Rkm$ by the induction hypothesis and
\[ I_h = I_{C_2} \subseteq I_{C_{2^n}} = \mathfrak{m},\]
the element ${t}_{2^{n-(i-1)}m}^{C_{2^{i-1}}}$ is invertible in $\Rkm$.

Now, for all $k$ such that $2^{n-(i-1)}m<k\leq h$, we have
\begin{eqnarray*}
t_k^{C_{2^{i-1}}}
&\equiv& t_k^{C_{2^i}} + \gamma_i t_k^{C_{2^i}} + \sum_{j=1}^{k-1}\gamma_{i}{t}_{j}^{C_{2^i}} ({t}_{k-j}^{C_{2^i}})^{2^j} \pmod{I_h}.
\end{eqnarray*}
Again, using the fact that the elements $C_{2^i} \cdot {t}_{k}^{C_{2^i}}$ are zero modulo $I_h$ for ${2^{n-i}m< k\leq h}$, every term in this sum vanishes modulo $I_h$.  This completes the induction step.
\end{proof}

By letting $i = 1$ in Proposition~\ref{prop:invertHHRelementspreprop} and using Proposition~\ref{prop:tkvk}, we obtain the following corollary.

\begin{cor}\label{cor:vhInvertible}
The element $v_h$ is invertible in $\Rkm$.
\end{cor}

\begin{proposition}\label{prop:m=Ih}
The ideals
$\mathfrak{m}$ and $I_h$ are equal in $\Rkm$.
\end{proposition}
\begin{proof}
By Proposition~\ref{prop:tkvk} and Proposition~\ref{prop:mandC2n}, it suffices to prove that $I_{C_2} = I_{C_{2^n}}$.
We will prove that $I_{C_2} = I_{C_{2^r}}$ for all $1 \leq r \leq n$ by using induction on $r$.  The base case, when $r =1$, is trivial.

Let $1<r\leq n$ and suppose we have shown that
$$I_{C_2} = I_{C_{2^{r-1}}}.$$
For simplicity of notations, let $k:= 2^{n-r}m$.  Consider the ideals
\begin{eqnarray*}
I_{C_{2^{r-1}}} &=& (2, C_{2^{r-1}} \cdot t_1^{C_{2^{r-1}}}, \ldots, C_{2^{r-1}} \cdot t_{2k-1}^{C_{2^{r-1}}}, C_{2^{r-1}} \cdot(t_{{2k}}^{C_{2^{r-1}}} - \gamma_{r-1} t_{2k}^{C_{2^{r-1}}})), \\
I_{C_{2^r}} &=& (2, C_{2^{r}} \cdot t_1^{C_{2^r}}, \ldots, C_{2^{r}} \cdot t_{k-1}^{C_{2^r}}, C_{2^{r}} \cdot(t_{k}^{C_{2^r}} - \gamma_r t_{{k}}^{C_{2^r}})).
\end{eqnarray*}
For $1 \leq i \leq k-1$, define $J_i$ to be the ideal
\begin{eqnarray*}
J_i &=& I_{C_{2^{r-1}}} + (2, C_{2^r} \cdot t_1^{C_{2^r}}, \ldots, C_{2^r} \cdot t_{i-1}^{C_{2^r}}) \\
&=& I_{C_2} + (2, C_{2^r} \cdot t_1^{C_{2^r}}, \ldots, C_{2^r} \cdot t_{i-1}^{C_{2^r}}).
\end{eqnarray*}
Note that the equality holds because of our inductive hypothesis.  Also, by Proposition~\ref{prop:danny2}, the ideal $J_i$ is $\gamma_r$-invariant and
\[ J_{i+1} = J_{i} +(C_{2^r} \cdot t_{i}^{C_{2^r}}) .\]
We will use downward induction on $i$ to show that the elements
$$\{C_{2^r} \cdot t_i^{C_{2^r}}, \ldots, C_{2^r} \cdot t_{k-1}^{C_{2^r}}\}$$
are in the ideal $J_i$ for all $1 \leq i \leq k-1$.  In particular, at $i=1$, this will imply that the elements
$$\{C_{2^r} \cdot t_1^{C_{2^r}}, \ldots, C_{2^r} \cdot t_{k-1}^{C_{2^r}}\}$$
are in the ideal $J_1 = I_{C_{2^{r-1}}} = I_{C_2}$.

The base case, when $i=k-1$, is proven as follows.  By Theorem~\ref{theorem:Mform}, we have the formulas
\begin{eqnarray*}
t_{k-1}^{C_{2^{r-1}}} &\equiv& t_{k-1}^{C_{2^r}} + \gamma_r t_{k-1}^{C_{2^r}} + \sum_{j=1}^{k-2}\gamma_{r}{t}_{j}^{C_{2^r}} ({t}_{k-1-j}^{C_{2^r}})^{2^j} \pmod{I_{C_2}} \\
&\equiv& t_{k-1}^{C_{2^r}} + \gamma_r t_{k-1}^{C_{2^r}} \pmod{J_{k-1}}
\end{eqnarray*}
and
\begin{eqnarray*}
t_{2k-1}^{C_{2^{r-1}}} &\equiv& \gamma_r t_{k-1}^{C_{2^r}}(t_k^{C_{2^r}})^{2^{k-1}} + \gamma_r t_{k}^{C_{2^r}}(t_{k-1}^{C_{2^r}})^{2^k} \pmod{I_{C_2}}.
\end{eqnarray*}
Since $t_{k-1}^{C_{2^{r-1}}} \in J_{k-1}$, the first equation implies that
$$t_{k-1}^{C_{2^r}} \equiv \gamma_r t_{k-1}^{C_{2^r}} \pmod{J_{k-1}}.$$
Substituting this into the second equation and using the fact that $t_{2k-1}^{C_{2^{r-1}}} \in I_{C_2}$ from Proposition~\ref{prop:invertHHRelementspreprop} yields the relation
\begin{eqnarray*}
t_{k-1}^{C_{2^r}}(t_k^{C_{2^r}})^{2^{k-1}} + \gamma_r t_{k}^{C_{2^r}}(t_{k-1}^{C_{2^r}})^{2^k} &\equiv& 0 \pmod{J_{k-1}} \\
\Longrightarrow t_{k-1}^{C_{2^r}}\left((t_k^{C_{2^r}})^{2^{k-1}}+ \gamma_r t_{k}^{C_{2^r}}(t_{k-1}^{C_{2^r}})^{2^k-1}\right) &\equiv& 0 \pmod{J_{k-1}}.
\end{eqnarray*}
By Proposition~\ref{prop:invertHHRelementspreprop}, the element $(t_k^{C_{2^r}})^{2^{k-1}}$ is a unit in $\Rkm$.  Since
\[\gamma_r t_{k}^{C_{2^r}}(t_{k-1}^{C_{2^r}})^{2^k-1} \in \mathfrak{m},\] the sum
$$(t_k^{C_{2^r}})^{2^{k-1}}+ \gamma_r t_{k}^{C_{2^r}}(t_{k-1}^{C_{2^r}})^{2^k-1}$$
is a unit in $\Rkm$.  Therefore,
$$t_{k-1}^{C_{2^r}} \equiv 0 \pmod{J_{k-1}}.$$
By Proposition~\ref{prop:danny2}, the ideal $J_{k-1}$ is $\gamma_r$-invariant.  It follows from this that all the elements $C_{2^r} \cdot t_{k-1}^{C_{2^r}}$ are in $J_{k-1}$.  This proves the base case of the induction.

Suppose we have proven the claim for $i+1 \leq k-1$.  To prove the claim for $i$, it suffices to show that the elements $C_{2^r} \cdot t_i^{C_{2^r}}$ are in the ideal $J_i$.  Once we have established this, it will follow from the induction hypothesis that all the elements in
$$\{C_{2^r} \cdot t_i^{C_{2^r}}, \ldots, C_{2^r} \cdot t_{k-1}^{C_{2^r}}\}$$
are also in the ideal $J_i$.  Indeed, the induction hypothesis implies that the elements
$$\{C_{2^r} \cdot t_{i+1}^{C_{2^r}}, \ldots, C_{2^r} \cdot t_{k-1}^{C_{2^r}}\}$$
are in the ideal $J_{i+1}$ and $J_{i+1} = J_i + (C_{2^r} \cdot t_{i}^{C_{2^r}})$.

By Theorem~\ref{theorem:Mform}, we have
\begin{eqnarray*}
t_{i}^{C_{2^{r-1}}} &\equiv& t_{i}^{C_{2^r}} + \gamma_r t_{i}^{C_{2^r}} + \sum_{j=1}^{i-1} \gamma_r t^{C_{2^r}}_{j}(t_{i-j}^{C_{2^r}})^{2^{j}}\pmod{I_{C_2}} \\
&\equiv& t_{i}^{C_{2^r}} + \gamma_r t_{i}^{C_{2^r}} \pmod{J_i}.
\end{eqnarray*}
Since $t_{i}^{C_{2^{r-1}}} \in J_i$, this implies that
 \begin{equation}\label{eq:thefirstequation}
 t_i^{C_{2^r}} \equiv \gamma_r t_i^{C_{2^r}} \pmod{J_i}.
 \end{equation}
By Theorem~\ref{theorem:Mform} again, we have
\begin{eqnarray*}
t_{k+i}^{C_{2^{r-1}}} &\equiv&   t_{k+i}^{C_{2^r}} + \gamma_r t_{k+i}^{C_{2^r}} + \sum_{j=1}^{k+i-1} \gamma_r t^{C_{2^r}}_{j}(t_{k+i-j}^{C_{2^r}})^{2^{j}}\pmod{I_{C_2}} \\
&\equiv&  \gamma_r t_i^{C_{2^r}}(t_k^{C_{2^r}})^{2^i} + \gamma_r t_{i+1}^{C_{2^r}}(t_{k-1}^{C_{2^r}})^{2^{i+1}} + \cdots + \gamma_r t_k^{C_{2^r}}(t_i^{C_{2^r}})^{2^k} \pmod{J_i},
\end{eqnarray*}
where the second equality uses Proposition~\ref{prop:invertHHRelementspreprop}.
Since the induction hypothesis implies that $t_{k+i}^{C_{2^{r-1}}} \in I_{C_2}$, this gives
\begin{equation}\label{eq:thesecondquation}
0 \equiv  \gamma_r t_i^{C_{2^r}}(t_k^{C_{2^r}})^{2^i} + \gamma_r t_{i+1}^{C_{2^r}}(t_{k-1}^{C_{2^r}})^{2^{i+1}} + \cdots + \gamma_r t_k^{C_{2^r}}(t_i^{C_{2^r}})^{2^k}\pmod{J_i} .
\end{equation}
Substituting Equation~\eqref{eq:thefirstequation} into Equation~\eqref{eq:thesecondquation}, we obtain the equality
\begin{eqnarray*}
0 &\equiv& t_{i}^{C_{2^r}}(t_k^{C_{2^r}})^{2^i} + \gamma_r t_{i+1}^{C_{2^r}}(t_{k-1}^{C_{2^r}})^{2^{i+1}} + \cdots + \gamma_r t_k^{C_{2^r}}(t_i^{C_{2^r}})^{2^k} \pmod{J_i}  \\
&\equiv& t_{i}^{C_{2^r}}(t_k^{C_{2^r}})^{2^i} + (t_{i}^{C_{2^r}})^2 \cdot x \pmod{J_i} \hspace{0.2in} \text{(by induction hypothesis)}\\
&\equiv& t_{i}^{C_{2^r}} \left((t_k^{C_{2^r}})^{2^i} +t_{i}^{C_{2^r}} \cdot x \right)\pmod{J_i}\\
&\equiv& t_{i}^{C_{2^r}} \cdot \text{unit} \pmod{J_i}.
\end{eqnarray*}
Here,
\[x\equiv (t_i^{C_{2^r}})^{-2}\left( \gamma_r t_{i+1}^{C_{2^r}}(t_{k-1}^{C_{2^r}})^{2^{i+1}} + \cdots + \gamma_r t_k^{C_{2^r}}(t_i^{C_{2^r}})^{2^k}\right) \pmod{J_i}.\]
This makes sense because each of the elements in $C_{2^r}\cdot t_{i+1}^{C_{2^r}}$, $\ldots$, $C_{2^r}\cdot t_{k-1}^{C_{2^r}}$ is divisible by $t_i^{C_{2^r}}$ modulo $J_i$. Indeed, $t_i^{C_{2^r}} \equiv \gamma_r t_i^{C_{2^r}} $ modulo $J_i$ as shown above and the elements $C_{2^r}\cdot t_{i+1}^{C_{2^r}}$, $\ldots$, $C_{2^r}\cdot t_{k-1}^{C_{2^r}}$ are in $J_{i+1}$ by the induction hypothesis. So,
\[J_{i+1} = J_i + (C_{2^r} \cdot t_{i}^{C_{2^r}}) \equiv ( t_i^{C_{2^r}}) \pmod{J_i}.\]
The last equality holds because $t_k^{C_{2^r}}$ is a unit in $\pi_* E_h$ and $t_{i}^{C_{2^r}} \cdot x \in \mathfrak{m}$.  This implies that $t_i^{C_{2^r}} \equiv 0 \pmod{J_i}$.  Since the ideal $J_i$ is $\gamma_r$-invariant by Proposition~\ref{prop:danny2}, all of the elements $C_{2^r} \cdot t_i^{C_{2^r}}$ are in $J_i$.  This finishes the induction step.

When $i=1$, the elements
$$\{C_{2^r} \cdot t_1^{C_{2^r}}, \ldots, C_{2^r} \cdot t_{k-1}^{C_{2^r}}\}$$
are all in $J_1 = I_{C_2}$.  Applying Theorem~\ref{theorem:Mform} produces the relation
\begin{eqnarray*}
t_{k}^{C_{2^{r-1}}} &\equiv& t_k^{C_{2^r}} + \gamma_r t_k^{C_{2^r}} + \sum_{j=1}^{k-1} \gamma_r t_{j}^{C_{2^r}} (t_{k-j}^{C_{2^r}})^{2^{j}} \pmod{I_{C_2}}.
\end{eqnarray*}
Therefore,
$$0 \equiv t_k^{C_{2^r}} + \gamma_n t_k^{C_{2^r}} \pmod{I_{C_2}},$$
and the elements $C_{2^r} \cdot (t_k^{C_{2^r}} - \gamma_r t_k^{C_{2^r}})$ are in $I_{C_2}$.  It follows that $I_{C_{2^r}} = I_{C_2}$.  This completes the induction step.
\end{proof}

\begin{theorem}\label{thm:mainalg}
The formal group law $(K,\bar u, \Gamma_h)$ of Equation~\ref{eq:fglgammah} has height $h$. Furthermore, the formal group law $(\Rkm ,u ,F_h)$ is a universal deformation of $(K,\bar u, \Gamma_h)$ and
\[  \Gkm \subseteq \G(k, \Gamma_h), \]
where $\Gkm$ is defined as in \eqref{eq:Gkm}.
\end{theorem}
\begin{proof}
The ring $\Rkm$ has Krull dimension $h$.  In particular, the regular sequence of elements
\[\{C_{2^n} \cdot t_1^{C_{2^n}}, \ldots, C_{2^n} \cdot t_{m-1}^{C_{2^n}}, C_{2^n} \cdot (u - \gamma_n u) \}\]
in $\Rkm$ forms a generating set for $\mathfrak{m}$.  Since $I_h = \mathfrak{m}$ by Proposition~\ref{prop:m=Ih} and $I_h$ is generated by the $h$ elements
$\{2, v_1, \ldots, v_{h-1} \}$,
these elements also form a regular sequence in $\Rkm$ that generates the maximal ideal $\mathfrak{m}$.

By Corollary~\ref{cor:vhInvertible}, the element $v_h$ is a unit in $\Rkm$.  This shows that
\[\Gamma_h := p^*F_h\]
is a formal group law of height $h$ over the residue field $\Rkm/\mathfrak{m} = K$. We conjugate $F_h$ and $\Gamma_h$ by $u$ to obtain formal group laws $F_h^0$ and $\Gamma_h^0$ over $\Rkm_0$, the homogenous elements of degree zero, and $k$ respectively. Let $\mathfrak{m}_0 = \Rkm_0 \cap \mathfrak{m}$ and $u_i =v_iu^{1-2^i} $ in $\Rkm_0$. The map 
\[W(k)[\![u_1, \ldots, u_{h-1}]\!] \to \Rkm_0 \]
is an isomorphism, as can be verified by filtering both sides by the maximal ideal $(2,u_1, \ldots, u_{h-1}) = \mathfrak{m}_0$. Because of its relationship to $v_i$, the element $u_i$ is  by definition the coefficients of $x^{2^i}$ in the $2$-series of $F_h^0$ modulo $(2, u_1, \ldots, u_{i-1})$. It follows that $(\Rkm_0, F_h^0)$ satisfies all the conditions of \cite[Proposition 1.1]{lubintate}, and so is a universal deformation for $(k,\Gamma_h^0)$.

Finally, since the action of $\Gkm$ on $\Rkm$ is faithful and via continuous ring isomorphisms, $\Gkm \subseteq \G(k, \Gamma_h)$ by \eqref{eq:morcont}.
\end{proof}

This concludes the algebra needed to establish Theorem~\ref{thm:model}.

\section{An equivariant Lubin--Tate spectrum}\label{sec:proofs-top}
In this section, we turn to study the Lubin--Tate spectrum $E(k,\Gamma_h)$ and prove Theorems~\ref{thm:model} and \ref{thm:IntroEquivOrientation}.  The universal deformation $F_h$ of $\Gamma_h$ studied in the previous section defines
 a Lubin--Tate spectrum $E(k,\Gamma_h)$.
By the Goerss--Hopkins--Miller theorem, the action of $\Gkm$ on $\Gamma_h$ gives rise to an action of $\Gkm$ on our Lubin--Tate theory $E(k,\Gamma_h)$ by maps of $E_{\infty}$-ring spectra.  We then promote our spectrum $E(k,\Gamma_h)$ to a $C_{2^n}$-spectrum and show that as a $C_{2^n}$-spectrum, $E(k, \Gamma_h)$ has an equivariant orientation in the sense that there is a $C_{2^n}$-equivariant map
\[\MUCn \longrightarrow E(k, \Gamma_h)\]
that classifies $F_h$ on underlying homotopy groups.

We will also prove that the homotopy fixed point spectrum $E(k, \Gamma_h)^{h\Ckm}$ of $E(k, \Gamma_h)$ by a subgroup $\Ckm \subset \mathbb{G}(k, \Gamma_h)$ of order coprime to 2 also admits a $C_{2^n}$-equivariant orientation.

\subsection{The classical Lubin--Tate spectrum}\label{sec:classLT}

To obtain $E(k,\Gamma_h)$, we simply apply the Landweber exact functor theorem and the Goerss--Hopkins--Miller theorem. Combined, these give the following result.
\begin{theorem}\label{thm:maintop}
There is a complex orientable $E_{\infty}$-ring spectrum $E(k,\Gamma_h)$ such that $\pi_*E(k,\Gamma_h) =\Rkm$. The spectrum $E(k,\Gamma_h)$ has a continuous action of $\G(k,\Gamma_h)$ by maps of $E_{\infty}$-ring spectra which refines the action of $\G(k,\Gamma_h)$ on $\Rkm$.
\end{theorem}

This finishes the proof of Theorem~\ref{thm:model}. We will now give a description of $\pi_*E(k,\Gamma_h)$ which emphasizes the structure of $\pi_0$, as mentioned in Remark~\ref{rem:pi0remark}.
\begin{proposition}\label{thm:modelE0-alg}
There are elements
\[C_{2^n}\cdot \tau_i = \{\tau_i, \gamma_n \tau_i, \ldots, \gamma_{n}^{2^{n-1}-1}\tau_i \} \subseteq \Rkm_0\]
for $1\leq i\leq m-1$ and
 \[\overline{C_{2^n}\cdot \tau_m} = \{\tau_m, \gamma_n \tau_m, \ldots, \gamma_{n}^{2^{n-1}-2}\tau_m \} \subseteq \Rkm_0\]
 (note that there is no generator ``$\gamma_{n}^{2^{n-1}-1}\tau_m $'')
such that
\[ \Rkm \cong  W(k) [\![ C_{2^n}\cdot \tau_1, \ldots, C_{2^n}\cdot \tau_{m-1},\overline{C_{2^n}\cdot \tau_m} ]\!][u^{\pm1}] .\]
The $C_{2^n}$-action on $\Rkm$ is determined by the formula
$$f_{\gamma_n} (\gamma_n^r x) = \gamma_n^{r+1} x $$
for $x = \tau_i$ ($1 \leq i \leq m$) and $r\leq 2^{n-1}-2$. Furthermore
\begin{enumerate}
\item for $1\leq i\leq m$,
$f_{\gamma_n} ( \gamma_n^{2^{n-1}-1} \tau_i) =\tau_i$,
\item
$$f_{\gamma_n}(\gamma_n^{2^{n-1}-2}\tau_m)= 1+\frac{1}{(1-\tau_m)(1-\gamma_n\tau_m)\ldots (1-\gamma_n^{2^{n-1}-2} \tau_m)}.$$
\item For $1\leq r \leq 2^{n-1}-1$,
\begin{align*}
f_{\gamma_n^r}( u) =   (1-\gamma^{r-1} \tau_m)(1-\gamma^{r-2}\tau_m)\ldots (1-\tau_m ) u
\end{align*}
and $f_{\gamma_n^{2^{n-1}}}(u) = -u$.
\end{enumerate}
The group $\Gal(k/\F_2)$ acts on $\Rkm$ via its action on the coefficients $W(k)$, and the action of $\zeta \in k^{\times}[q]$ fixes $\tau_m$ and is determined by
\begin{eqnarray*}
f_\zeta(u) &=& \zeta^{-1}u, \\
f_\zeta(\tau_{i}) &=& \zeta^{2^i-1}\tau_i
\end{eqnarray*}
for $1\leq i \leq m-1$.
\end{proposition}

\begin{proof}
Let $\gamma=\gamma_n$.
For $1\leq i <m$ and $0\leq r \leq 2^{n-1}-1$, let $\gamma^r\tau_i = \gamma^r(t_i^{C_{2^n}}u^{1-2^i})$. Then
\begin{align*}
\gamma (\gamma^{r-1}\tau_i ) = \begin{cases}  \gamma^r\tau_i & 0<r<2^{n-1}, \\
\tau_i & r=2^{n-1}.
\end{cases}
\end{align*}
For $0\leq r \leq 2^{n-1}-2$, let
 $\gamma^r(\tau_m)=\gamma^r(1-u^{-1}\gamma u)$. Clearly, for $0< r \leq 2^{n-1}-2$, $\gamma (\gamma^{r-1}\tau_m) = \gamma^{r}\tau_m$. Furthermore,
 \[\gamma(\gamma^{2^{n-1}-2} \tau_m) = 1+ \frac{u}{\gamma^{2^{n-1}-1} u}.\]
Since
 \[\frac{1}{1-\gamma^r \tau_m} = \frac{\gamma^r u}{\gamma^{r+1} u}\]
 for $0\leq r \leq 2^{n-1}-2$, we conclude that
  \[\gamma(\gamma^{2^{n-1}-2} \tau_m)  = 1+\frac{1}{(1-\tau_m)\cdots (1-\gamma^{2^{n-1}-2} \tau_m)},\]
  which proves the claim in the statement of the theorem. The action of $\Ckm$ is clear from the definition of the $\tau_i$s.
\end{proof}

Theorem~\ref{thm:maintop} implies that $E(k,\Gamma_h)$ has an action of $\Gkm$ by maps of $E_{\infty}$-ring spectra.   Before promoting $E(k,\Gamma_h)$ to an equivariant spectrum, we prove the following splitting result.

\begin{theorem}\label{thm:thm4.8}
Let
\begin{equation*}
\Ckm =  \Gal \ltimes k^{\times}[q] \subseteq \Gkm.\end{equation*}
There is a $\Gkm$-equivariant map
\[   E(k, \Gamma_h) \to E(k,\Gamma_h)^{h\Ckm }\]
which splits the natural map $E(k,\Gamma_h)^{h\Ckm } \to E(k,\Gamma_h)$.
\end{theorem}
\begin{proof}
In this proof, let $E=E(k,\Gamma_h)$, $\G =\G(k,\Gamma_h)$ and $\mathbb{S} = \mathbb{S}(k,\Gamma_h)$.
Note that the group $k^{\times}[q] $ is cyclic.  Let
 $\ell$ be its order and $\zeta$ be a generator so that $\zeta^\ell=1$. Define
\[\varepsilon := \frac{1}{q} \sum_{i=0}^{\ell-1} [\zeta^i] .\]
Letting $[\zeta^i]$ act on $E$ via the action of $k^{\times}[q]$, we obtain a map
\[E \stackrel{\varepsilon}{\longrightarrow} E.\]
 Let $\varepsilon^{-1}E$ be the telescope of $\varepsilon$. Note that since any element of $\Gal$ permutes the set $\{\zeta^i\}_{i=0}^{\ell-1}$, $\varepsilon$ commutes with the action of $\Gal$.  Similarly, $\varepsilon$ commutes with the action of $C_{2^n}$ and $k^{\times}[q]$.  Therefore, $\Gkm$ acts on $\varepsilon^{-1}E$ and the map
 \[E \to \varepsilon^{-1}E\]
is  $\Gkm$-equivariant. Furthermore, the composite
\[ E^{hk^{\times}[q]} \to E \to \varepsilon^{-1}E \]
is a $\Gkm$-equivariant map which is an isomorphism on homotopy groups. This can be verified by using the collapse of the homotopy fixed points spectral sequence for $E^{hk^{\times}[q]}$. Therefore, the composite is a $\Gkm$-equivariant equivalence.

Now, note that
\[E^{h\Ckm }  \simeq (E^{hk^{\times}[q] })^{h\Gal} .\]
By \cite[Lemma 1.37]{BobkovaGoerss}, there is a $\Gal$-equivariant equivalence
\[\Gal_+ \smsh E^{h\Ckm} \to  E^{hk^{\times}[q] .}\]
This is shown by first proving that the composite
\begin{equation}\label{eq:longeq}
E^{h\mathbb{S}} \smsh E^{h\Ckm} \to E^{hk^{\times}[q ]} \smsh E^{hk^{\times}[q]} \to  E^{hk^{\times}[q]} \end{equation}
obtained by the natural maps followed by multiplication
is a weak equivalence, and then appealing to the $\Gal$-equivariant equivalence
\[ \Gal_+ \smsh E^{h\mathbb{G}}  \to E^{h\mathbb{S}} \]
proved in \cite[Lemma 1.36]{BobkovaGoerss}. However, note that the latter map is a $\Gkm$-equivalence if we equip both spectra with trivial $C_{2^n} \times k^{\times}[q]$-actions. Furthermore, \eqref{eq:longeq} is also a $\Gkm$-equivariant map. Therefore, the equivalence
\[\Gal_+ \smsh E^{h\Ckm} \simeq  E^{hk^{\times}[q]}\] is $\Gkm$-equivariant. It follows that $E^{h\Ckm}$ splits off equivariantly from $E^{hk^{\times}[q]}$, hence from $E$.
\end{proof}

\subsection{$E(k, \Gamma_h)$ as an equivariant spectrum}\label{sec:equivspectrum}
We will now upgrade $E(k, \Gamma_h)$ to a commutative $C_{2^n}$-spectrum.  By Theorem~\ref{thm:maintop}, we may view $E(k, \Gamma_h)$ as a commutative ring object in naive $C_{2^n}$-spectra.  The functor
$$X \longmapsto F(E{C_{2^n}}_+, X)$$
takes naive equivalences to genuine equivariant equivalences, and hence allows us to view $E(k, \Gamma_h)$ as a genuine $C_{2^n}$-equivariant spectrum.

The commutative ring spectrum structure on $E(k, \Gamma_h)$ gives rise to an action of a trivial $E_{\infty}$-operad on $F(E{C_{2^n}}_+, E(k, \Gamma_h))$.  Work of Blumberg--Hill \cite{BlumbergHill} shows that this is sufficient to ensure that $F(E{C_{2^n}}_+, E(k, \Gamma_h))$ is a genuine equivariant commutative ring spectrum (see also \cite[Section 2.2]{hill_meier}).  Therefore, by passing to the cofree localizations, we may view $E(k, \Gamma_h)$ and $E(k, \Gamma_h)^{h\Ckm}$ as commutative $C_{2^n}$-spectra.

\begin{proof}[Proof of Theorem~\ref{thm:IntroEquivOrientation}]
The Real orientation theorem of \cite{hahnshi} implies that the complex orientation
\[MU=i^*_e\MUR \to i^*_eE(k,\Gamma_h)\]
refines to a Real orientation
\[\MUR \to i^*_{C_2}E(k, \Gamma_h).\]
The $2$-typical nature of our formal group laws imply that these maps factor through $BP = i^*_e\BPR$ and $\BPR$ respectively.

Applying the norm functor to the maps
$$\MUR \longrightarrow \BPR \longrightarrow i^*_{C_2}E(k, \Gamma_h)$$
and post-composing with the counit map of the norm-restriction adjunction gives maps
$$\MUCn \longrightarrow \BPCn \longrightarrow  N_{C_2}^{C_{2^n}}  i^*_{C_2}E(k, \Gamma_h) \longrightarrow E(k, \Gamma_h)$$
of $C_{2^n}$-ring spectra.  Consider the composite
\[\phi: \BPCn \longrightarrow  N_{C_2}^{C_{2^n}}  i^*_{C_2}E(k, \Gamma_h) \longrightarrow E(k, \Gamma_h). \]
By construction, $\pi_*^e\phi$ is the map
$f \colon \Rn \to \Rkm  $
defined in \eqref{eq:mapf}, which is the same map as in the statement of Theorem~\ref{thm:IntroEquivOrientation}.
\end{proof}

In fact, we can obtain a refinement of Theorem~\ref{thm:IntroEquivOrientation}.  There is a similar $C_{2^n}$-equivariant map from $\MUCn$ to the fixed point spectrum $E(k, \Gamma_h)^{h\Ckm}$, where $\Ckm \subset \mathbb{G}(k, \Gamma_h)$ is the subgroup defined in Theorem~\ref{thm:thm4.8}.  Suppose $k^{\times}[q]$ has $\alpha$ elements, where $1\leq \alpha \leq 2^m-1=q$.  It follows from the description of the $\Gkm$-action on $\pi_*E(k,\Gamma_h)$ (Theorem~\ref{thm:model}) that
\[
(\pi_*E(k,\Gamma))^{\Ckm} \cong \mathbb{Z}_2 [C_{2^n} \cdot t_1^{C_{2^n}}, \ldots, C_{2^n} \cdot t_{m-1}^{C_{2^n}}, C_{2^n}\cdot u^{\alpha}][C_{2^m} \cdot (u^{\alpha})^{-1}]^{\wedge}_{\mathfrak{m'}},
\]
where
$$\mathfrak{m}'=(C_{2^n}\cdot t_1^{C_{2^n}}, \ldots, C_{2^n}\cdot t_{m-1}^{C_{2^n}}, C_{2^n}\cdot (u^\alpha-\gamma_n u^\alpha)).$$
If $\alpha=q$, then $u^\alpha = u^{2^m-1} = t_m^{C_{2^n}}$ and
\[
(\pi_*E(k,\Gamma))^{\Ckm} \cong \mathbb{Z}_2 [C_{2^n} \cdot t_1^{C_{2^n}}, \ldots, C_{2^n} \cdot t_{m-1}^{C_{2^n}}, C_{2^n}\cdot t_m^{C_{2^n}}][C_{2^m} \cdot (t_m^{C_{2^n}})^{-1}]^{\wedge}_{\mathfrak{m'}},
\]
where
$$\mathfrak{m}'=(C_{2^n}\cdot t_1^{C_{2^n}}, \ldots, C_{2^n}\cdot t_{m-1}^{C_{2^n}}, C_{2^n}\cdot (t_{m}^{C_{2^n}}-\gamma_n t_{m}^{C_{2^n}})).$$

Furthermore,  the homotopy fixed points spectral sequence
\[
E_2^{s,t} = H^s(\Ckm, \pi_t E(k,\Gamma_h)) \Longrightarrow  \pi_{t-s} E(k,\Gamma_h)^{h\Ckm}
\]
has the property that $E_2^{>0,*} =0$. This follows from the fact that the action of $\Gal$ on $\pi_*E(k,\Gamma)$ is free, and that the order of $k^{\times}[q]$ is odd.  Therefore, the homotopy fixed points spectral sequence collapses and we have the following result.
\begin{proposition}\label{prop:CqFixedPoint}
There is an isomorphism
\[
\pi_*E(k,\Gamma_h)^{h\Ckm} \cong \mathbb{Z}_2 [C_{2^n} \cdot t_1^{C_{2^n}}, \ldots, C_{2^n} \cdot t_{m-1}^{C_{2^n}}, C_{2^n}\cdot u^{\alpha}][C_{2^m} \cdot (u^{\alpha})^{-1}]^{\wedge}_{\mathfrak{m'}},
\]
where
\[\mathfrak{m}'=(C_{2^n}\cdot t_1^{C_{2^n}}, \ldots, C_{2^n}\cdot t_{m-1}^{C_{2^n}}, C_{2^n}\cdot (u^\alpha-\gamma_n u^\alpha)).\]
If $\alpha=q$, then $u^\alpha=u^{2^m-1} = t_m^{C_{2^n}}$.
\end{proposition}

On the underlying homotopy groups, the map
$$\pi_*^eE(k, \Gamma_h) \longrightarrow \pi_*^eE(k, \Gamma_h)^{h\Ckm}$$
is not a ring map, but it is a $C_{2^n}$-equivariant map that sends $t_i^{C_{2^n}} \mapsto t_i^{C_{2^n}}$ for $1 \leq i \leq m-1$, and $u^{\alpha} \mapsto u^\alpha$.

For simplicity, for our next theorem we will choose our field $k$ so that $k^\times[q]$ has $(2^m-1)$-elements.

\begin{theorem}\label{thm:EquivOrientationCq}
There is a $C_{2^n}$-equivariant homotopy commutative ring map
$$\MUCn \longrightarrow E(k, \Gamma_h)^{h\Ckm}.$$
This map factors through a homotopy commutative ring map
$$\psi: \BPCn \longrightarrow E(k, \Gamma_h)^{h\Ckm}$$
such that the map $\pi_*^e \psi$ is the map $\Rn \longrightarrow \pi_*^e E(k, \Gamma_h)^{h\Ckm}$ determined by
\begin{align*}
t_i^{C_{2^n}} &\longmapsto \begin{cases}t_i^{C_{2^n}} & 1\leq i \leq m, \\
0 & i>m.
\end{cases}
\end{align*}
\end{theorem}
\begin{proof}
Consider the splitting map
$${E(k, \Gamma_h) \longrightarrow E(k,\Gamma_h)^{h\Ckm}}$$
in Theorem~\ref{thm:maintop}.  Although this map is not a ring map, it is still a $C_{2^n}$-equivariant map and hence induces a map
\begin{equation}\label{eq:EtoCq}
C_2\text{-}\HFPSS(E(k, \Gamma_h)) \longrightarrow C_2\text{-}\HFPSS(E(k, \Gamma_h)^{h\Ckm})
\end{equation}
of $C_2$-equivariant homotopy fixed points spectral sequences (HFPSS).

On the other hand, we also have the map
$$E(k, \Gamma_h)^{h\Ckm} \longrightarrow E(k, \Gamma_h),$$
which is a map of commutative $C_{2^n}$-spectra.  This map induces a map
\begin{equation}\label{eq:CqtoE}
C_2\text{-}\HFPSS(E(k, \Gamma_h)^{h\Ckm}) \longrightarrow C_2\text{-}\HFPSS(E(k, \Gamma_h)).
\end{equation}
The composition map of spectral sequences
$$C_2\text{-}\HFPSS(E(k, \Gamma_h)^{h\Ckm}) \longrightarrow C_2\text{-}\HFPSS(E(k, \Gamma_h)) \longrightarrow C_2\text{-}\HFPSS(E(k, \Gamma_h)^{h\Ckm})$$
is the identity map.

By Theorem~\ref{thm:model}, and \cite[Theorem~4.7]{hill_meier}, the $E_2$-page of the $RO(C_2)$-graded $C_2$-homotopy fixed points spectral sequence for $E(k, \Gamma_h)$ is
$$W(k) [C_{2^n} \cdot \bar{t}_1^{C_{2^n}}, \ldots, C_{2^n} \cdot \bar{t}_{m-1}^{C_{2^n}}, C_{2^{n}} \cdot \bar{u} ] [C_{2^{n}}\cdot \bar{u}^{-1} ]^{\wedge}_{\bar{\mathfrak{m}}}\otimes \mathbb{Z}[u_{2\sigma}^\pm, a_\sigma]/(2a_\sigma).$$
By Proposition~\ref{prop:CqFixedPoint}, the $E_2$-page of the $RO(C_2)$-graded $C_2$-homotopy fixed points spectral sequence for $E(k, \Gamma_h)^{h\Ckm}$ is
$$\mathbb{Z}_2 [C_{2^n} \cdot \bar{t}_1^{C_{2^n}}, \ldots, C_{2^n} \cdot \bar{t}_{m-1}^{C_{2^n}}, C_{2^n}\cdot \bar{u}^{\alpha}][C_{2^m} \cdot (\bar{u}^{\alpha})^{-1}]^{\wedge}_{\bar{\mathfrak{m}}'}\otimes \mathbb{Z}[u_{2\sigma}^\pm, a_\sigma]/(2a_\sigma).$$

The map~\eqref{eq:CqtoE} induces an injection on the $E_2$-page.  Hahn--Shi \cite[Theorem~1.2]{hahnshi} have completely computed all the differentials in $C_2\text{-}\HFPSS(E(k, \Gamma_h))$.  By natuality of the maps \eqref{eq:EtoCq} and \eqref{eq:CqtoE}, we deduce that the map
$$C_2\text{-}\HFPSS(E(k, \Gamma_h)^{h\Ckm}) \longrightarrow C_2\text{-}\HFPSS(E(k, \Gamma_h))$$
also induces injections on the set of differentials on each page.  More specifically, for any nonzero differential $d_r(x) = y$ in $C_2\text{-}\HFPSS(E(k, \Gamma_h)^{h\Ckm})$, its image in $C_2\text{-}\HFPSS(E(k, \Gamma_h))$ is also the nonzero differential $d_r(x) = y$.

As a consequence, we deduce that
$$\pi_{k\rho_2-1}^{C_2}(E(k, \Gamma_h)^{h\Ckm}) = 0$$
for all $k \in \mathbb{Z}$ (this is because $\pi_{k\rho_2-1}^{C_2}(E(k, \Gamma_h)) = 0$ for all $k \in \mathbb{Z}$).  By \cite[Lemma~3.3]{hill_meier}, the spectrum $i_{C_2}^*E(k, \Gamma_h)^{h\Ckm}$ is Real orientable, and we obtain a homotopy commutative ring map
$$\MUR \longrightarrow i_{C_2}^*E(k, \Gamma_h)^{h\Ckm} $$
that factors through $\BPR$.

Since $E(k, \Gamma_h)^{h\Ckm}$ is a $C_{2^n}$-equivariant commutative ring, applying the norm functor $N_{C_2}^{C_{2^n}}(-)$ to
$$\MUR \longrightarrow \BPR \longrightarrow i_{C_2}^*E(k, \Gamma_h)^{h\Ckm} $$
and using the norm-forget adjunction produces the homotopy commutative ring maps
$$\begin{tikzcd}
\MUCn \ar[r] & \BPCn \ar[r, "\psi"] &E(k, \Gamma_h)^{h\Ckm}.
\end{tikzcd}$$

The map $\pi_*^e\psi$ is determined by the $C_{2^n}$-action on the formal group law over $\pi_*^eE(k, \Gamma_h)^{h\Ckm}$ defined via the map
$$BP_* \longrightarrow \pi_*^eE(k, \Gamma_h)^{h\Ckm}.$$
By construction, it is the map we claimed in the statement of the theorem.
\end{proof}

\section{Equivariant orientation and localization}\label{sec:localization}
In this section, we prove Theorem~\ref{thm:IntroInvertingD}.  Throughout this section we will denote the group $C_{2^n}$ by $G$ and the Lubin--Tate theory $E(k, \Gamma_h)$ by $E_h$. We let $\rho_G$ be the real regular representation of $G$ and we abbreviate $\rho_{2} = \rho_{C_2}$.
We need to specify an element $D \in \pi_{*\rho_G}^G \MUCn$ so that there are the desired factorizations as stated in the theorem and the following three properties hold:
\begin{enumerate}
\item The spectra $D^{-1} \MUG$, $D^{-1}\BPG$ are cofree.
\item The Hill--Hopkins--Ravenel periodicity theorem \cite[Theorem 9.19]{HHR} holds for $D^{-1} \MUG$ and $D^{-1}\BPG$.
\item In $\pi_*^e D^{-1}\BPG \langle m \rangle$, $I_{C_2} = I_{G}$, where
\begin{eqnarray*}
I_{C_2} &=& (2, t_1^{C_2}, \ldots, t_{2^{n-1}m-1}^{C_2}, 2 t_{2^{n-1}m}^{C_2}) = (2, v_1, \ldots, v_{2^{n-1}m-1})\\
I_G &=& (2, G \cdot t_1^{G}, \ldots, G \cdot t_{m-1}^G, G \cdot (t_m^G - \gamma_n t_m^G))
\end{eqnarray*}
are the ideals defined in the proof of Proposition~\ref{prop:m=Ih}.
\end{enumerate}

Before specifying the element $D$ so that properties (1)--(3) hold, we will first explain how to obtain the factorizations in Theorem~\ref{thm:IntroInvertingD} once we have chosen an arbitrary element $D \in \pi_\bigstar^G \MUG$ that becomes invertible in $\pi_\bigstar^G E_h$.

Given a homotopy commutative spectrum $R$, the spectrum $D^{-1}R$ is defined to be the homotopy colimit of the sequence
$$\begin{tikzcd}
R \ar[r, "D"] &S^{-V} \wedge R \ar[r, "D"] &S^{-2V} \wedge R \ar[r, "D"] &\cdots.
\end{tikzcd}$$
The $C_{2^n}$-equivariant orientation
$$\MUG \longrightarrow E_h$$
is a map of homotopy commutative ring spectra, and there is a commutative diagram
$$\begin{tikzcd}
\MUG \ar[r, "D"] \ar[d] & S^{-V} \wedge \MUG \ar[r, "D"] \ar[d] & S^{-2V} \wedge \MUG \ar[r, "D"] \ar[d] & \cdots \\
E_h \ar[r, "D"] & S^{-V} \wedge E_h \ar[r, "D"] & S^{-2V} \wedge E_h \ar[r, "D"] & \cdots
\end{tikzcd}$$
Passing to the colimit and using the fact that $D^{-1}E_h \simeq E_h$ produces the factorization map
$$D^{-1}\MUG \longrightarrow E_h.$$
This proves the first diagram.  The proof for factorization through $D^{-1}\BPG$ is exactly the same.

\begin{remark}\rm
As we will see, the element $D \in \pi_{*\rho_G}^G \MUG$ also becomes invertible in $\pi_\bigstar^G E_h^{h\Ckm}$ under the map
$$\pi_{\bigstar}^G \MUG \longrightarrow \pi_{\bigstar}^G E_h^{h\Ckm}.$$
It follows from Theorem~\ref{thm:IntroInvertingD} and the discussion above that there are factorizations
$$\begin{tikzcd}
\MUG \ar[d] \ar[r] &E_h^{h\Ckm} \\
D^{-1} \MUG \ar[ru, dashed]
\end{tikzcd} \hspace{0.3in} \begin{tikzcd}
\BPG \ar[d] \ar[r] &E_h^{h\Ckm} \\
D^{-1} \BPG \ar[ru, dashed]
\end{tikzcd}$$
for the $G$-equivariant orientations of $E_h^{h\Ckm}$ through $D^{-1} \MUG$ and $D^{-1}\BPG$.
\end{remark}

We will now specify the element $D \in \pi_{\bigstar}^G \MUG$ so that Theorem~\ref{thm:IntroInvertingD} holds. By \cite[Section 5]{HHR} and \cite[Theorem 6.7]{hahnshi}, the spectra $i_{C_2}^*\MUG$, $i_{C_2}^*\BPG$, $i_{C_2}^*\BPG\langle m\rangle$ and $i_{C_2}^*E_h$ are strongly even, which means in particular that the restriction maps
$$\pi_{*\rho_2}^{C_2}(-) \longrightarrow \pi_{*}^e(-)$$
from the $(*\rho_2)$-graded $C_2$-equivariant homotopy groups to the non-equivariant homotopy groups are isomorphisms.  Therefore, we have complete knowledge of the homotopy groups of $\pi_{*\rho_2}^{C_2} \MUG$, $\pi_{*\rho_2}^{C_2} \BPG$, $\pi_{*\rho_2}^{C_2} \BPG\langle m\rangle$, and $\pi_{*\rho_2}^{C_2} E_h$.  They are
\begin{eqnarray*}
\pi_{*\rho_2}^{C_2} \MUG &=& \mathbb{Z}[G \cdot \bar{r}_1^G, G \cdot \bar{r}_2^G, \ldots],\\
\pi_{*\rho_2}^{C_2} \BPG &=& \mathbb{Z}_{(2)}[G \cdot \bar{t}_1^G, G \cdot \bar{t}_2^G, \ldots],\\
\pi_{*\rho_2}^{C_2} \BPG\langle m\rangle &=& \mathbb{Z}_{(2)}[G \cdot \bar{t}_1^G, G \cdot \bar{t}_2^G, \ldots, G \cdot \bar{t}_m^G],\\
\pi_{*\rho_2}^{C_2} E_h &=& \mathbb{Z}_2 [G \cdot \bar{t}_1^{G}, \ldots, G \cdot \bar{t}_{m-1}^{G}, G \cdot \bar{u} ] [G\cdot \bar{u}^{-1} ]^{\wedge}_{\bar{\mathfrak{m}}}
\end{eqnarray*}
where
\begin{align*} \bar{\mathfrak{m}} &=(C_{2^n}\cdot \bar{t}_1^{C_{2^n}}, \ldots, C_{2^n}\cdot \bar{t}_{m-1}^{C_{2^n}}, C_{2^n}\cdot (\bar{u}-\gamma_n \bar{u}))\\
 &=(C_{2^n}\cdot \bar{t}_1^{C_{2^n}}, \ldots, C_{2^n}\cdot \bar{t}_{m-1}^{C_{2^n}}, C_{2^n}\cdot (\bar{t}_m-\gamma_n \bar{t}_m)).
\end{align*}

The following proposition gives a criterion to identify elements in $\pi_{*\rho_G}^G \MUG$ that becomes invertible under the induced map
$$\pi_{*\rho_G}^G \MUG \longrightarrow \pi_{*\rho_G}^G E_h$$
of $G$-equivariant homotopy groups.  The same result holds for $\BPG$ as well.

\begin{proposition}\label{prop:NormingRho2}
If the element $x \in \pi_{*\rho_2}^{C_2} \MUG$ becomes invertible under the map
$$\pi_{*\rho_2}^{C_2} \MUG \longrightarrow \pi_{*\rho_2}^{C_2} E_h$$
of $C_2$-equivariant homotopy groups, then the element $N_{C_2}^{G}(x) \in \pi_{*\rho_{G}}^G \MUG$ also becomes invertible under the map
$$\pi_{*\rho_{G}}^{G} \MUG \longrightarrow \pi_{*\rho_{G}}^{G} E_h$$
of $G$-equivariant homotopy groups.
\end{proposition}
\begin{proof}
Let the image of $x$ under the map
$$\pi_{*\rho_2}^{C_2} \MUG \longrightarrow \pi_{*\rho_2}^{C_2} E_h$$
be $y$.  We will prove that the image of $N_{C_2}^G(x)$ under the map
$$\pi_{*\rho_{G}}^{G} \MUG \longrightarrow \pi_{*\rho_{G}}^{G} E_h$$
is $N_{C_2}^G(y)$, which is invertible.

We will denote the slice spectral sequence and the homotopy fixed points spectral sequence by $\SliceSS(-)$ and $\HFPSS(-)$, respectively.  Consider the maps
$$C_2\text{-}\SliceSS(\MUG) \longrightarrow C_2\text{-}\HFPSS(\MUG) \longrightarrow C_2\text{-}\HFPSS(E_h)$$
of $RO(C_2)$-graded spectral sequences.  The element $x$ is represented by a class on the $E_2$-page of $C_2\text{-}\SliceSS(\MUG)$, which, by an abuse of notation, will also be denoted by $x$.  On the $E_2$-page, the maps of spectral sequences above send
$$x \longmapsto x' \longmapsto y,$$
where $x'$ and $y$ are classes on the $E_2$-page of $C_2\text{-}\HFPSS(\MUG)$ and the $E_2$-page of $C_2\text{-}\HFPSS(E_h)$, respectively.  Since $x$ is a permanent cycle, $x'$ and $y$ are also permanent cycles.  The element $y$ survives to become the element which, again, we also call $y \in \pi_{*\rho_2}^{C_2} E_h$ in homotopy.

Now, consider the maps
$$G\text{-}\SliceSS(\MUG) \longrightarrow G\text{-}\HFPSS(\MUG) \longrightarrow G\text{-}\HFPSS(E_h)$$
of $RO(G)$-graded spectral sequences.  On the $E_2$-page, the class $N_{C_2}^G (x)$ is first mapped to $N_{C_2}^G(x')$, and then mapped to $N_{C_2}^G(y)$.  This is because the second map on the $E_2$-page is completely determined by the $G$-equivariant map
$$\pi_*^e\MUG \longrightarrow \pi_*^eE_h.$$
The classes $N_{C_2}^G (x)$, $N_{C_2}^G (x')$, and $N_{C_2}^G (y)$ are permanent cycles, and they all survive to the $E_\infty$-page.  It follows that as elements in the $G$-equivariant homotopy groups $\pi_{*\rho_G}^G(-)$,
$$N_{C_2}^G (x) \longmapsto N_{C_2}^G (y)$$
under the map
\[\pi_{*\rho_{G}}^{G} \MUG \longrightarrow \pi_{*\rho_{G}}^{G} E_h.\qedhere\]
\end{proof}

We will now specify the element $D \in \pi_{\bigstar}^G \MUG$ so that properties (1)--(3) hold.  Our method is as follows: first, we will identify elements $x \in \pi_{*\rho_2}^{C_2} \BPG$ that become invertible under the map
$$\pi_{*\rho_2}^{C_2} \BPG \longrightarrow \pi_{*\rho_2}^{C_2}E_h.$$
By Proposition~\ref{prop:NormingRho2}, the elements $N_{C_2}^{G}(x) \in \pi_{*\rho_G}^G \BPG$ (formed by considering $x$ as elements in $\pi_{*\rho_2}^{C_2} \MUG$ under the map $\pi_{*\rho_2}^{C_2}\BPG \longrightarrow \pi_{*\rho_2}^{C_2}\MUG$) will also become invertible in $\pi_{*\rho_G}^G E_h$ under the map
$$\pi_{*\rho_G}^{G} \BPG \longrightarrow \pi_{*\rho_G}^{G}E_h.$$
We will define
$$D := \prod N_{C_2}^G(x)$$
to be the product of the elements $N_{C_2}^G(x)$.

\begin{proposition}\label{prop:invertHHRelements}
The images of the elements $\bar{t}_{2^{n-i}m}^{C_{2^i}} \in \pi_{*\rho_2}^{C_2} \BPG$ for $1 \leq i \leq n$ are invertible in $\pi_{*\rho_2}^{C_2} E_h$.
\end{proposition}
\begin{proof}
This is an immediate consequence of Proposition~\ref{prop:invertHHRelementspreprop}.
\end{proof}

\begin{proof}[Proof of Theorem~\ref{thm:IntroInvertingD}.]
Proposition~\ref{prop:invertHHRelements} shows that we can include the product
\begin{align}\label{eq:defDini}
\prod_{i =1}^n N_{C_2}^G(\bar{t}_{2^{n-i}m}^{C_{2^i}}) = N_{C_2}^G(\bar{t}_{2^{n-1}m}^{C_2}) \cdot N_{C_2}^G(\bar{t}_{2^{n-2}m}^{C_4}) \cdots N_{C_2}^G(\bar{t}_{m}^{C_{2^n}})
\end{align}
into $D$.  By the arguments in \cite[Section 9]{HHR} and \cite[Section 10]{HHR}, inverting these elements will produce periodicity and homotopy fixed points theorems for the spectra $D^{-1} \MUG$, $D^{-1}\BPG$, and their quotients.  Therefore properties (1) and (2) hold.

Now, we will include elements into $D$ so that property (3) holds.  We will describe an iterative algorithm to accomplish this.  In the proof of Proposition~\ref{prop:m=Ih}, we used induction on $n$ to show that $I_{C_2} = I_{C_{2^n}}$ in $\pi_*E_h$.  For each step of the induction process, we defined intermediate ideals $J_i \subset \pi_*^e \BPG\langle m \rangle$ and used downward induction on $i$ to show that certain elements are in the images of ideals $J_i$ in $\pi_*E_h$.

In the argument of the downward induction, we identified certain elements in $\pi_*^e \BPG\langle m \rangle$ that become invertible in $\pi_*E_h$.  For instance, the elements
$${(t_k^{C_{2^r}})^{2^{k-1}}+ \gamma_r t_{k}^{C_{2^r}}(t_{k-1}^{C_{2^r}})^{2^k-1}}$$
and
$$(t_k^{C_{2^r}})^{2^i} +t_{i}^{C_{2^r}} \cdot x$$
are such elements (see the proof of Proposition~\ref{prop:m=Ih}).  Our algorithm is as follows: everytime we identify such an element $t \in \pi_*^e \BPG\langle m \rangle$, include $N_{C_2}^G (\bar{t})$ into the product defining $D$, where $\bar{t} \in \pi_{*\rho_2}^{C_2} \BPG\langle m \rangle$ is the (unique) $C_2$-equivariant lift of $t$.  For the two elements mentioned above, we will include
$$N_{C_2}^G((\bar{t}_k^{C_{2^r}})^{2^{k-1}}+ \gamma_r \bar{t}_{k}^{C_{2^r}}(\bar{t}_{k-1}^{C_{2^r}})^{2^k-1})$$
and
$$N_{C_2}^G((\bar{t}_k^{C_{2^r}})^{2^i} +\bar{t}_{i}^{C_{2^r}} \cdot \bar{x})$$
into the product.  Including all such elements to the product defining $D$ will guarantee that the proof for $I_{C_2} = I_{G}$ will carry through in $\pi_*^e D^{-1}\BPG\langle m \rangle$ as well.  This proves property (3).  \end{proof}

\section{The height of $BP^{(\!(C_{2^n})\!)}\langle m \rangle$}\label{sec:heightfilt}

In this section, we continue to let $G=C_{2^n}$ and $h=2^{n-1}m$. We now turn to analyze the height of the formal group law over $BP^{(\!(G)\!)}\langle m \rangle$.  We start by making a few remarks that will render this analysis easier. By Proposition~\ref{prop:tkvk}, we have
\[t_k^{C_2}\equiv v_k \pmod {I_k}\]
in $BP_*$.  In the equivalence above, the generators $v_k$ are the Araki generators. The generators ``$v_k$" are only well defined modulo $I_k$ and any choice of these generators will give the same chromatic story. So, instead of using classical choices of generators for $BP_*$ such as the Araki or Hazewinkel generators, we can use generators $t_k^{C_2}$ in our analysis of the heights of various $BP$-modules. For instance, since the Bousfield class of $E(h)$ is the same as that of $v_h^{-1}BP$ \cite[Theorem 7.3.2]{RavNil}, which is the same as that of $(t_h^{C_2})^{-1}BP$, we have
\[L_{h} X  = L_{(t_h^{C_2})^{-1}BP}X\]
for any spectrum $X$.

For this reason, from now on, we redefine
\begin{align}
\label{eq:redefinevk} v_k := t_k^{C_2} \in BP_*.
\end{align}
As usual, we let $I_r = (2,v_1, \ldots, v_{r-1})$.  As an immediate consequence of our work in the previous sections, we have the following result.  Note that in the introduction, we called
\[ \Rnm:= \pi_*^e i^*_e \BPG\langle m \rangle.\]
\begin{proposition}\label{prop:IhIGBP}
In $\pi_*^e D^{-1}\BPG\langle m \rangle$, $I_h = I_G$ and $\pi_*^e D^{-1}\BPG\langle m \rangle$ is a regular local ring with maximal ideal $I_h$ generated by the regular sequence $(2, v_1, \ldots, v_{h-1})$. Furthermore, $v_h$ maps to $t_m^\beta$ in
\[ \pi_*^e D^{-1}\BPG\langle m \rangle/I_h \cong  \F_2[t_m^{\pm 1}],\]
where $\beta=(2^h-1)/(2^m-1)$. In particular, modulo $I_h$, the formal group law over $\pi_{*}^{e}D^{-1}BP^{(\!(G)\!)}\langle m\rangle$ has height exactly $h$.
\end{proposition}
\begin{proof}
In the proof of Theorem~\ref{thm:IntroInvertingD}, we show that $I_h = I_G$. Now, as in the proof of Theorem~\ref{thm:mainalg}, we get that $(2, v_1, \ldots, v_{h-1})$ is a regular sequence by analyzing the Krull dimension of $\pi_*^e D^{-1}\BPG\langle m \rangle$.
Finally, $v_h$ (which is $t_{h}^{C_2}$) is a factor in $i^*_eD$ (see Equation~\ref{eq:defDini}). Therefore, $v_h$ is a unit and so maps to a unit in $\pi_*^e D^{-1}\BPG \langle m \rangle /I_h$.  The identification of $ \pi_*^e D^{-1}\BPG\langle m \rangle/I_h$ is straightforward by using the fact that $I_G=I_h$.  It follows by degree reasons that $v_h$ maps to $t_m^\beta$.
\end{proof}

Note that, as we have mentioned in the proof above, the element $v_h$ is invertible in $\pi_* i^*_e D^{-1}BP^{(\!(G)\!)}$ because it is a factor in $i^*_eD$.  Therefore $i^*_e D^{-1}BP^{(\!(G)\!)}$ is a $v_h^{-1}BP$-module and so is $E(h)$-local.

\begin{lemma}\label{lem:heightBPbracketsm}
For all $0 \leq k \leq n-1$ and $r > 2^k m$, under the composite map
\[\pi_*^e BP^{(\!(C_{2^{n-k}} )\!)} \longrightarrow \pi_*^e \BPCn \longrightarrow \pi_*^e \BPCn \langle m \rangle,
\]
the images of $t_r^{C_{2^{n-k}}}$ and its conjugates by $C_{2^{n-k}}$ are contained in the ideal
${I_r = (2, v_1, \ldots, v_{r-1})}$.
\end{lemma}
\begin{proof}
We will use induction on $k$.  The base case when $k = 0$ is immediate because for $r > m$, under the map
\[\pi_*^e \BPCn \longrightarrow  \pi_*^e \BPCn \langle m \rangle,\]
$t_r^{C_{2^n}}$ and its conjugates by $C_{2^n}$ are all sent to 0.

Now, suppose the claim is true for $k$.  By Theorem~\ref{theorem:Mform}, we have the following equality in $\pi_*^e \BPCn \langle m \rangle$ modulo $I_r$:
\[t_r^{C_{2^{n-k-1}}} \equiv t_r^{C_{2^{n-k}}} + \gamma_{n-k} t_r^{C_{2^{n-k}}} + \sum_{j = 1}^{r - 1} (t_{r-j}^{C_{2^{n-k}}})^{2^j} \gamma_{n-k} t_j^{C_{2^{n-k}}}.\]
For $r > 2^{k+1} m$, every summand on the right hand side contains some $t_i^{C_{2^{n-k}}}$ or its conjugate with $i \geq r/2 > 2^k m$.  By the induction hypothesis, these elements are all in the ideal $I_i \subseteq I_r  $.  It follows that $t_r^{C_{2^{n-k-1}}} = 0$ modulo $I_r$ for $r > 2^{k+1}m$.  The same proof applies to its conjugates.  This finishes the induction.
\end{proof}

\begin{proposition}\label{prop:heightBPbracketsm} Let $h=2^{n-1}m$.  Under the composite map
\[
\pi_* BP \longrightarrow \pi_*^e \BPG  \longrightarrow \pi_*^e \BPG \langle m \rangle,
\]
the images of the $v_i$ generators satisfy $v_r \in (2, v_1, \ldots, v_h)$ for $r >h$.
\end{proposition}
\begin{proof}
Set $k = n-1$ in Lemma~\ref{lem:heightBPbracketsm}.  The result of the lemma implies that for all $r > h$, $v_r = t_r^{C_2}$ is contained in the ideal $I_r = (2, v_1, \ldots, v_{r-1})$.  In other words, $I_{r+1} = I_r$.  Applying the lemma iteratively shows that
\[I_{r+1} = I_r = \cdots = I_{h+1} = (2, v_1, \ldots, v_h).\]
It follows that $v_r \in (2, v_1, \ldots, v_h)$, as desired.
\end{proof}

\begin{proposition}\label{prop:computingKr} \hfill
\begin{enumerate}[(1)]
\item For $0 \leq r \leq h$,
\[\pi_*L_{K(r)} i^*_e D^{-1}\BPG \langle m \rangle \cong \left(v_r^{-1}\pi_*i^*_e  D^{-1} \BPG \langle m \rangle \right)^{\wedge}_{I_r}, \]
and $L_{K(r)} i^*_{e} D^{-1}BP^{(\!(G)\!)}\langle m\rangle \not\simeq \ast$.
\item For $r > h$, $L_{K(r)} i^*_{e} D^{-1}BP^{(\!(G)\!)}\langle m\rangle  \simeq \ast $.
\end{enumerate}
\end{proposition}
\begin{proof}
In  \cite[Section 4]{HovStrick}, the authors produce a cofinal sequence $J(i) = (j_0, j_1, \ldots, j_{r-1})$ of positive integers and generalized Moore spectra
\[M_{J(i)} =S^0/(v_0^{j_0},  \ldots, v_{r-1}^{j_{r-1}})   \]
with maps $M_{J(i+1)} \to M_{J(i)}$ so that, for any spectrum $X$,
\[ L_{K(r)}X \simeq \mathrm{holim}_i M_{J(i)} \smsh L_r X.\]
This gives a $\lim^1$-sequence
\[ 0 \to {\lim_i}^{1} \pi_{*+1}(M_{J(i)} \smsh L_r X) \to \pi_*L_{K(r)}X \to \lim_i  \pi_{*}(M_{J(i)} \smsh L_r X) \to 0.\]

We apply this to $X=i^*_e D^{-1}BP^{(\!(G)\!)} \langle m \rangle $.  First, we show that $L_rX = v_r^{-1} X$.  Note that $v_r^{-1}X$ is $E(r)$-local because $L_rBP = v_r^{-1}BP$, and $v_r^{-1}X$  is a $v_r^{-1}BP$-module.  For any $E(r)$-local spectrum $Y$ and a map $X \to Y$, we get a map $v_r^{-1}BP \to Y$ from the composition map $BP \to X \to Y$.  This implies that the map $X \to Y$ factors through the map $X \to v_r^{-1}X$.  It follows from the universal property of $L_rX$ that $L_r X \simeq v_r^{-1}X$.

We can obtain $M_{J(i)} \smsh L_r X$ by a series of cofiber sequences
\[\Sigma^{j_{k}|v_k|}S^0/(v_0^{j_0},  \ldots, v_{k-1}^{j_{k-1}}) \smsh L_r  X \xrightarrow{v_{k}^{j_k}} S^0/(v_0^{j_0},  \ldots, v_{k-1}^{j_{k-1}}) \smsh L_r X \to S^0/(v_0^{j_0}, \ldots, v_{k}^{j_{k}}) \smsh L_r X  .\]
We start with the case when $r\leq h$. Since the sequence $(v_0, v_1, \ldots, v_{r-1})$ is regular in $\pi_* L_r X$, so is the sequence $(v_0^{j_0}, v_1^{j_1}, \ldots, v_{r-1}^{j_{r-1}})$. It follows that we get a series of exact sequences
\[ 0 \to \pi_*L_r X/(v_0^{j_0},  \ldots, v_{k-1}^{j_{k-1}})   \xrightarrow{v_{k}^{j_k}} \pi_*L_r X/(v_0^{j_0},  \ldots, v_{k-1}^{j_{k-1}})  \to \pi_*L_r X/(v_0^{j_0},  \ldots, v_{k}^{j_{k}})  \to 0, \]
which lead to an isomorphism
\[ \pi_{*}(M_{J(i)} \smsh L_r X) \cong (\pi_*L_r X )/(v_0^{j_0},  \ldots, v_{r-1}^{j_{r-1}})  . \]
 The maps in the inverse system $ \lim_i  \pi_{*}(M_{J(i)} \smsh L_r X)$ are then obviously surjective and so ${\lim_i}^{1} \pi_{*+1}(M_{J(i)} \smsh L_r X) =0$. The exact sequence above gives an isomorphism
\[  \pi_*L_{K(r)}X \xrightarrow{\cong}  \lim_i  \pi_{*}(M_{J(i)} \smsh L_r X) \cong (v_r^{-1}\pi_*X)^{\wedge}_{I_r}.\]

To show that this is not 0, note that since $(2, v_1, \ldots, v_r)$ is a regular sequence, $v_r \colon \pi_*X/I_r^k \to \pi_*X/I_r^k$ is injective.  Therefore, $\pi_*X/I_r^k$ (which is clearly non-zero) injects into  $v_r^{-1}\pi_*X/I_r^k$. It follows that $\lim_k \pi_*X/I_r^k $ injects into $\lim_k v_r^{-1}\pi_*X/I_r^k$, and so the latter is nontrival. This proves (1).

For (2), note that $v_h$ is in $D$ by definition.  In the series of cofiber sequences forming $M_{J(i)} \wedge L_rX$, the first map below induces an equivalence:
\[\Sigma^{j_{h}|v_h|}S^0/(v_0^{j_0},  \ldots, v_{h-1}^{j_{h-1}}) \smsh L_r  X \xrightarrow{v_{h}^{j_h}} S^0/(v_0^{j_0},  \ldots, v_{h-1}^{j_{h-1}}) \smsh L_r X \to S^0/(v_0^{j_0}, \ldots, v_{h}^{j_{h}}) \smsh L_r X.\]
This implies that $S^0/(v_0^{j_0}, \ldots, v_{h}^{j_{h}}) \smsh L_r X = 0$ and therefore $M_{J(i)} \wedge L_rX = 0$.  It follows that every term in the tower $\{M_{J(i)} \wedge L_r X\}$ is contractible and ${L_{K(r)} X \simeq \ast}$.
\end{proof}

\begin{theorem}\label{cor:computingKr} \hfill
\begin{enumerate}
\item For $0 \leq r \leq h$, $L_{K(r)} i_e^* \BPG \langle m \rangle \not \simeq \ast$.
\item For $r > h$, $L_{K(r)} i_e^* \BPG \langle m \rangle \simeq \ast$.
\end{enumerate}
\end{theorem}
\begin{proof}
For (1), note that the maps
\[S^0 \longrightarrow L_{K(r)} i_e^* \BPG \langle m \rangle \longrightarrow L_{K(r)} i_e^* D^{-1} \BPG \langle m \rangle\]
are ring maps.  Since $L_{K(r)} i_e^* D^{-1} \BPG \langle m \rangle \not \simeq \ast$ by Proposition~\ref{prop:computingKr}, it follows that ${L_{K(r)} i_e^* \BPG \langle m \rangle \not \simeq \ast}$.

For (2), let $X = i_e^* \BPG\langle m \rangle$.  We will show that $ L_r X \wedge M_{J(i)} \simeq \ast$ for the generalized Moore spectra
\[M_{J(i)} = S^0/(v_0^{j_0}, \ldots, v_{r-1}^{j_{r-1}}).\]
This will imply every term in the tower $\{L_r X \wedge M_{J(i)} \}$ is contractible and ${L_{K(r)} X \simeq \ast}$.  Note that
$$L_r X \wedge M_{J(i)} = L_rX \wedge_{MU} MU/(v_0^{j_0}, \ldots, v_{r-1}^{j_{r-1}}). $$
There is a K\"unneth spectral sequence \cite[Theorem IV.4.1]{EKMM}
\[E_2^{s,t} = \Tor_{-s, t}^{MU_*}(\pi_*X, \pi_* MU/(v_0^{j_0}, \ldots, v_{r-1}^{j_{r-1}})) \Longrightarrow \pi_{t-s} \left(X \wedge_{MU} MU/(v_0^{j_0}, \ldots, v_{r-1}^{j_{r-1}})\right).\]
We have graded the spectral sequence cohomologically. As such, it is a lower half-plane spectral sequence.
Note that
\[ \pi_* MU/(v_0^{j_0}, \ldots, v_{r-1}^{j_{r-1}}) \cong MU_*/(v_0^{j_0}, \ldots, v_{r-1}^{j_{r-1}})\]
and the $E_2$-page
is a module over
\[\Tor_{0,*}^{MU_*}(\pi_* X, MU_*/(v_0^{j_0}, \ldots, v_{r-1}^{j_{r-1}})) = (\pi_*X) /(v_0^{j_0}, \ldots, v_{r-1}^{j_{r-1}}). \]

Since $v_r \in (v_0, \ldots, v_{r-1})$ by Proposition~\ref{prop:heightBPbracketsm}, $v_r^q \in (v_0^{j_0}, \ldots, v_{r-1}^{j_{r-1}})$ for ${q = \sum_{i=0}^{r-1} j_i}$.  This implies that $v_r^q$ is zero on the $E_2$-page.  Any element in the homotopy groups of $X \wedge_{MU} MU/(v_0^{j_0}, \ldots, v_{r-1}^{j_{r-1}})$ is represented by some element of filtration $s\leq 0$ on the $E_\infty$-page of the K\"unneth spectral sequence.  Since $v_r^q$ is zero on the $E_2$-page, this element must be annihilated by $v_r^{q(s+1)}$ in homotopy.  Therefore, every element in the homotopy groups of $X \wedge_{MU} MU/(v_0^{j_0}, \ldots, v_{r-1}^{j_{r-1}})$ is killed by some finite power of $v_r$.  It follows that
\[\pi_*(L_rX \wedge _{MU} MU/(v_0^{j_0}, \ldots, v_{r-1}^{j_{r-1}})) = v_r^{-1} \pi_* (X \wedge_{MU} MU/(v_0^{j_0}, \ldots, v_{r-1}^{j_{r-1}})) = 0.\qedhere\]
\end{proof}

\begin{proposition}
Let $q= 2^m -1$.  If $\F_{q} \subseteq k$ then the natural map $i^*_e BP^{(\!(G)\!)} \to i^*_e E(k, \Gamma_h)^{h\Ckm}$ of Theorem~\ref{thm:EquivOrientationCq} factors through an equivalence
\[ L_{K(h)} (i_e^* D^{-1}\BPG \langle m \rangle) \xrightarrow{\simeq} i^*_eE(k, \Gamma_h)^{h\Ckm}.\]
\end{proposition}
\begin{proof}
Let $E_h = E(k,\Gamma_h)$. The complex orientation $BP \to i^*_eE_h$ is a map of $A_{\infty}$-algebras and therefore so is the map
$i^*_e\BPG \to i^*_e E_h$.
It follows that $ i^*_eE_h$ is a $i^*_e\BPG$-module. Constructing $ i^*_e\BPG  \langle m \rangle $ as the quotient
\[ i^*_e\BPG/( t_{m+1}, \gamma t_{m+1}, \ldots, \gamma^{2^{n-1}-1} t_{m+1},   t_{m+2}, \gamma t_{m+2}, \ldots, \gamma^{2^{n-1}-1} t_{m+2}, \ldots)\]
via a series of cofiber sequences
and noting that $\gamma^i t_k$ maps to zero for $k\geq m+1$, we get a factorization
\[\xymatrix{ i^*_e\BPG \ar[r] \ar[d] & i^*_eE_h \\
 i^*_e\BPG  \langle m \rangle \ar@{.>}[ru] &  }\]
Composing the dotted arrow with the splitting of Theorem~\ref{thm:thm4.8} gives a map
\[ i^*_e\BPG  \langle m \rangle\to  i^*_eE_h^{h\Ckm}.  \]
Since $D$ is mapped to a unit in $\pi_*^eE_h$, this dotted arrow factors as a map
\[ i^*_eD^{-1}\BPG  \langle m \rangle\to  i^*_eE_h^{h\Ckm}.  \]
We apply the functor $L_{K(h)}(-)$ to this map. Since the target is already $K(h)$-local, we obtain a map
\[ \varphi \colon L_{K(h)}(i^*_eD^{-1}\BPG  \langle m \rangle)\to  i^*_eE_h^{h\Ckm}.  \]
It suffices to prove that $\varphi$ induces an isomorphism on homotopy groups.

In Proposition~\ref{prop:computingKr}, we proved that
\begin{align*}
\pi_* L_{K(h)}(i^*_eD^{-1}\BPG \langle m \rangle) \cong (\pi_*^eD^{-1}\BPG \langle m \rangle)^{\wedge}_{I_h}.
\end{align*}
Proposition~\ref{prop:IhIGBP} implies that
$I_h = I_{G}$ in $\pi_*^eD^{-1}\BPG$. By Proposition~\ref{prop:CqFixedPoint}, we have an isomorphism
\[
\pi_*E(k,\Gamma_h)^{h\Ckm} \cong \mathbb{Z}_2 [G \cdot t_1^{G}, \ldots, G \cdot t_{m-1}^{G}, G\cdot t_m][C_{2^m} \cdot (t_m)^{-1}]^{\wedge}_{I_{G}}.
\]
Here, we have used the fact that $\mathfrak{m}'=I_{G}$ if $t_m$ is a unit and that $\F_q \subseteq k$ (so that $\alpha=2^m-1$ in our application of Proposition~\ref{prop:CqFixedPoint}). Furthermore, by design, $\varphi$ maps $t_i$ and its conjugates in $\pi_*^eD^{-1}\BPG\langle m \rangle$ to the same named generators in $ \pi_*^eE(k,\Gamma_h)^{h\Ckm} $. Therefore, $\varphi$ induces an isomorphism on homotopy groups.
\end{proof}

\bibliographystyle{plain}
\bibliography{E-bib}

\end{document}